\documentclass[12pt,final]{scrartcl} 

\usepackage{xcolor}
\usepackage[notref,notcite,color]{showkeys}  

\usepackage[utf8]{inputenc}
\usepackage[T1]{fontenc}
\usepackage{amsmath, amstext, amsfonts}

\usepackage{enumitem} 

\usepackage[thmmarks,amsmath]{ntheorem}  
\usepackage[colorlinks,bookmarks]{hyperref}

\setlist{itemsep=1pt,parsep=0pt,topsep=2pt,partopsep=0pt}  
\setenumerate{leftmargin=*} 

\theoremstyle{break}
\newtheorem{thm}                   {Theorem}
\newtheorem{lem}           [thm] {Lemma}   
   
\newtheorem{prop}     [thm] {Proposition}

\newtheorem{defi}      [thm] {Definition} 
\theoremstyle{plain}
\newtheorem{fact}            [thm] {Fact} 
\theoremstyle{nonumberbreak}

\newtheorem{lemNN}             {Lemma}
\theoremstyle{nonumberplain}
\newtheorem{claim}      {Claim}[thm]
\newtheorem{remark}     {Remark}

\makeatletter
\newcommand{\openbox}{\leavevmode
  \hbox to.77778em{%
  \hfil\vrule
  \vbox to.675em{\hrule width.6em\vfil\hrule}%
  \vrule\hfil}}

\newcommand{\proofname}{Proof}
\newcounter{proof}%
\newenvironment{proof}[1][]{
  \th@nonumberplain
  \def\theorem@headerfont{\itshape}%
  \normalfont
  \@thm{proof}{proof}{\proofname{} {#1}\unskip.}}%
  {\@endtheorem}
\makeatother

\newcommand{\openboxsmall}{\leavevmode
  \hbox to.57778em{%
  \hfil\vrule
  \vbox to.575em{\hrule width.5em\vfil\hrule}%
  \vrule\hfil}}

\newenvironment{claimproof}[1][Proof]{

  \begin{proof}[#1]
}{
  \end{proof}

}

\newcommand{\hide}[1]{}

\let\version\today

\newcommand{\makeversion}{
  \let\oldfootnote\thefootnote %
  \renewcommand{\thefootnote}{} %
  \footnote{\version} %
  \let\thefootnote\oldfootnote %
}

\renewcommand{\tilde}{\widetilde}

\newcommand{\field}[1]{\mathbb{#1}}
\newcommand{\N}{\field{N}}

\newcommand{\eps}{\epsilon}
\renewcommand{\epsilon}{\varepsilon}

\renewcommand{\phi}{\varphi}

\DeclareMathOperator{\dcup}{\ensuremath{\mathaccent\cdot \cup}}  
\newcommand{\deff}{\mathrel{\mathop:}=}
\newcommand{\ffed}{\mathrel{=\mathop:}}

\usepackage[overload]{textcase}
\newcommand{\SP}{\textsc{\MakeTextLowercase{S}}} 
\newcommand{\ORD}{\textsc{\MakeTextLowercase{O}}} 

\newcommand{\cB}{\mathcal{B}}
\newcommand{\cA}{\mathcal{A}}
\newcommand{\cR}{\mathcal{R}}

\newcommand{\cH}{\mathcal{H}}
\newcommand{\cI}{\mathcal{I}}
\newcommand{\cCo}{\mathcal{C}\!o}

\newcommand{\By}[2]{\overset{\mbox{\tiny{#1}}}{#2}}    
\newcommand{\ByRef}[2]{   \By{\eqref{#1}}{#2} }      

\newcommand{\leBy}[1]{    \By{#1}{\le} }

\newcommand{\lByRef}[1]{  \ByRef{#1}{<} }

\newcommand{\leByRef}[1]{ \ByRef{#1}{\le} }

\newcommand{\geByRef}[1]{ \ByRef{#1}{\ge} }

\DeclareMathOperator{\EE}{\mathbb{E}}
\DeclareMathOperator{\PP}{\mathbb{P}}
\DeclareMathOperator{\Left}{Left}
\DeclareMathOperator{\Right}{Right}

\newcommand{\Free}{\text{\it Free}}
\newcommand{\Gnp}[1][p]{\mathcal{G}(n,#1)}

\newcommand{\EMAIL}[1]{  \textit{E-mail}: \texttt{#1} } 

\title{An Extension of the Blow-up Lemma to arrangeable graphs
 \thanks{
    YK was partially supported by CNPq (Proc.~308509/2007-2, 484154/2010-9, 477203/2012-4).
    AT and AW were partially supported by DFG grant TA 309/2-2.
    This work was partially supported by the University of S\~{a}o Paulo, through the MaCLinC project, by the Department of Mathematics at the London School of Economics, and by the Zentrum Mathematik at Technische Universit\"at M\"unchen, through the Women for Math Science-programme. 
The cooperation of the four authors was supported by a joint CAPES-DAAD project (415/ppp-probral/po/D08/11629, Proj. no. 333/09).
 }
}

  \author{
    Julia B\"ottcher\thanks{%
      Department of Mathematics, London School of Economics, Houghton
      Street, London WC2A 2AE, UK
      \EMAIL{j.boettcher@lse.ac.uk}
    }
    \and Yoshiharu Kohayakawa\thanks{
      Instituto de Matem\'atica e Estat\'{\i}stica, Universidade de S\~ao Paulo, Rua do Mat\~ao 1010, 05508--090~S\~ao Paulo, Brazil
      \EMAIL{yoshi@ime.usp.br}
    }
    \and Anusch Taraz\thanks{%
      Zentrum Mathematik, Technische Universit\"at M\"unchen, 
      Boltzmannstra\ss{}e~3, D-85747 Garching bei M\"unchen, Germany
      \EMAIL{taraz|wuerfl@ma.tum.de}
    }
    \and Andreas W\"urfl\footnotemark[4]
  } 

\begin{document}

\maketitle

\hyphenation{pre-decessors}


\begin{abstract}
  The Blow-up Lemma established by Koml\'os, S\'ark\"ozy, and Szemer\'edi
  in 1997 is an important tool for the embedding of spanning subgraphs of bounded
  maximum degree. Here we prove several generalisations of this result
  concerning the embedding of $a$-arrangeable graphs, where a graph is called
  $a$-arrangeable if its vertices can be ordered in such a way that the
  neighbours to the right of any vertex $v$ have at most $a$ neighbours to
  the left of $v$ in total. Examples of arrangeable graphs include planar
  graphs and, more generally, graphs without a $K_s$-subdivision for
  constant~$s$. Our main result shows that $a$-arrangeable graphs
  with maximum degree at most $\sqrt{n}/\log n$ can be embedded into
  corresponding systems of super-regular pairs. This is optimal up to the
  logarithmic factor. 

  We also present two applications. We prove that any large
  enough graph~$G$ with minimum degree at least $\big(\frac{r-1}{r}+\gamma\big)n$
  contains an $F$-factor of every $a$-arrangeable $r$-chromatic graph~$F$
  with at most $\xi n$ vertices and maximum degree at most $\sqrt{n}/\log
  n$, as long as $\xi$ is sufficiently small compared to $\gamma/(ar)$. This
  extends a result of Alon and Yuster [J. Combin. Theory Ser. B 66(2),
  269--282, 1996].
  Moreover, we show that for constant~$p$ the random graph $\Gnp$ is universal for the
  class of $a$-arrangeable $n$-vertex graphs~$H$ of maximum degree at most
  $\xi n/\log n$, as long as $\xi$ is sufficiently small compared to $p/a$.
  \end{abstract}



\renewcommand{\theenumi}{\roman{enumi}}
\renewcommand{\labelenumi}{(\theenumi)}

\section{Introduction}

The last 15 years have witnessed an impressive series of results
guaranteeing the presence of spanning subgraphs in dense graphs. In this
area, the so-called \emph{Blow-up Lemma} has become one of the key
instruments. It emerged out of a series of papers by Koml\'os, S\'ark\"ozy,
and Szemer\'edi (see
e.g.~\cite{KSS95,KSS_pos,KSS97,KSS98,KSS_seyweak,KSS_sey,KSS_alon}) and
asserts, roughly spoken, that we can find bounded degree spanning subgraphs
in $\eps$-regular pairs. It was used for determining, among others,
sufficient degree conditions for the existence of
$F$-factors, Hamilton paths and cycles and their powers, spanning trees and
triangulations, and graphs of sublinear bandwidth in graphs, digraphs and
hypergraphs (see the survey~\cite{KueOstSurvey} for an excellent overview
of these and related achievements). 
In this way, the Blowb<>-up Lemma has reshaped extremal graph theory.

However, with very few exceptions, the embedded spanning subgraphs~$H$ considered
so far came from classes of graphs with constant maximum
degree, because the Blow-up Lemma requires the subgraph
it embeds to have constant maximum degree. In fact, the Blow-up Lemma
is usually the only reason why the proofs of the above mentioned results
only work for such subgraphs.

The central purpose of this paper is to overcome this obstacle. We shall
provide extensions of the Blow-up Lemma that can embed graphs whose degrees
are allowed to grow with the number of vertices. These versions
require that the subgraphs we embed are arrangeable.\footnote{We remark
  that it was already suggested in~\cite{Komlos99} to relax the maximum
  degree constraint to arrangeability.}  We will formulate them in the
following and subsequently present some applications.

\paragraph{Blow-up Lemmas.}

We first introduce some notation. Let~$G$,
$H$ and $R$ be graphs with vertex sets $V(G)$, $V(H)$, and
$V(R)=\{1,\dots,r\}\ffed[r]$. 
For $v\in V(G)$ and $S,U\subseteq V(G)$ we define $N(v,S)\deff N(v)\cap S$
and $N(U,S)=\bigcup_{v\in U} N(v,S)$. 
Let $A,B\subset V(G)$ be non-empty and disjoint, and let
$\eps,\delta\in[0,1]$. The \emph{density} of the pair $(A,B)$ is defined to be
$d(A,B):=e(A,B)/(|A||B|)$. The pair $(A,B)$ is \emph{$\eps$-regular}, if
$|d(A,B)-d(A',B')|\le\eps$ for all $A' \subseteq A$ and $B' \subseteq B$
with $|A'|\geq\eps|A|$ and $|B'|\geq\eps|B|$. An $\eps$-regular pair
$(A,B)$ is called \emph{$(\eps,\delta)$-regular}, if $d(A,B)\ge\delta$ and
\emph{$(\eps,\delta)$-super-regular}, if $|N(v,B)| \ge \delta|B|$ for
all $v\in A$ and $|N(v,A)|\ge \delta|A|$ for all $v\in B$.
We say that~$H$ has an \emph{$R$-partition} $V(H)=X_1\dcup\dots\dcup X_r$,
if for every edge $xy\in E(H)$ there are distinct $i,j\in[r]$ with $x\in
X_i$, $y\in X_j$ and $ij\in E(R)$.  $G$ has a \emph{corresponding}
\emph{$(\eps,\delta)$-super-regular $R$-partition} $V(G)=V_1\dcup\dots\dcup
V_r$, if $|V_i|=|X_i|=:n_i$ for all $i\in [r]$ and every pair $(V_i,V_j)$
with $ij\in E(R)$ is $(\eps,\delta)$-super-regular. In this case~$R$ is
also called the \emph{reduced graph} of the super-regular partition.
Moreover, these partitions are \emph{balanced} if $n_1\le n_2\le \dots \le
n_r\le n_1+1$. They are \emph{$\kappa$-balanced} if $n_j\le \kappa n_i$ for
all $i,j\in[r]$. The partition classes $V_i$ are also called \emph{clusters}.

With this notation, 
a simple version of the Blow-up Lemma of Koml\'os, S\'ark\"ozy, and Szemer\'edi~\cite{KSS97} 
can now be formulated as follows.

\begin{thm}[Blow-up Lemma~\cite{KSS97}] \label{thm:Blow-up:KSS}
Given a graph $R$ of order $r$ and positive parameters $\delta, \Delta$, there exists a positive $\eps = \eps(r,\delta, \Delta)$ such that the following holds. 
Suppose that $H$ and $G$ are two graphs with the same number of vertices, where $\Delta(H) \le \Delta$ and $H$ has a balanced $R$-partition, and $G$ has a corresponding $(\eps,\delta)$-super-regular $R$-partition. Then there exists an embedding of $H$ into $G$.
\end{thm}

We remark that R\"odl and Ruci\'nski~\cite{RodlRuci99} gave a different
proof for this result. In addition, Koml\'os,
S\'ark\"ozy, and Szemer\'edi~\cite{KSS98} gave an algorithmic proof.

Our first result replaces the restriction on the maximum degree of $H$ in
Theorem~\ref{thm:Blow-up:KSS} by a restriction on its
arrangeability. This concept was first introduced by Chen and Schelp
in~\cite{CheSch93}.

\begin{defi}[$a$-arrangeable] \label{def:arrangeable} 
  Let $a$ be an integer. A graph is called \emph{$a$-arrangeable} if its vertices
  can be ordered as $(x_1,\dots ,x_n)$ in such a way that
  $\big|N\big(N(x_i,\Right_i),\Left_i\big)\big| \le a$ for each $1 \le i
  \le n$, where $\Left_i = \{x_1, x_2,\dots, x_i \}$ and $\Right_i =
  \{x_{i+1}, x_{i+2},\dots, x_n\}$.
\end{defi}

Obviously, every graph $H$ with $\Delta(H)\le a$ is
$(a^2-a+1)$-arrangeable.  Other examples for arrangeable graphs are planar
graphs: Chen and Schelp showed that planar graphs are 761-arrangeable
\cite{CheSch93}; Kierstead and Trotter \cite{KieTro93} improved this to
10-arrangeable. In addition, R\"odl and Thomas \cite{RodTho97} showed that
graphs without $K_s$-subdivision are $s^8$-arrangeable.  On the other hand,
even 1-arrangeable graphs can have unbounded degree (e.g.~stars).

\begin{thm}[Arrangeable Blow-up Lemma] \label{thm:Blow-up:arr} Given a
  graph $R$ of order $r$, a positive real $\delta$ and a natural number $a$,
	there exists a positive real $\eps = \eps(r,\delta,a)$ such that the following holds. Suppose
  that $H$ and $G$ are two graphs with the same number of vertices, where
  $H$ is $a$-arrangeable, $\Delta(H) \le \sqrt{n}/\log n$ and $H$ has a
  balanced $R$-partition, and $G$ has a corresponding
  $(\eps,\delta)$-super-regular $R$-partition. Then there exists an
  embedding of $H$ into $G$.
\end{thm}

Koml\'os, S\'ark\"ozy, and Szemer\'edi proved that the Blow-up Lemma allows for
the following strengthenings that are useful in applications. We
allow the clusters to differ in size by a constant factor and
we allow certain vertices of $H$ to restrict their image in~$G$ to be taken
from an a priori specified set of linear size. However, in contrast to the original
Blow-up Lemma, we need to be somewhat more restrictive about the
image restrictions: We still allow linearly many vertices in each cluster
to have image restrictions, but now only a constant number of different
image restrictions is permissible in each cluster (we shall show in
Section~\ref{sec:opt} that this is best possible).  
In the following, we state an extended version of the Blow-up Lemma that 
makes this precise.

\begin{thm}[Arrangeable Blow-up Lemma, full version] 
\label{thm:Blow-up:arr:full} 
  For all $C,a,\Delta_R,\kappa \in \N$ and for all $\delta,c>0$ there exist
  $\eps,\alpha>0$ such that for every integer $r$ there is $n_0$ such that
  the following is true for every $n\ge n_0$.
  Assume that we are given 
  \renewcommand{\theenumi}{\alph{enumi}}
  \begin{enumerate}
    \item a graph~$R$ of order $r$ with $\Delta(R)<\Delta_R$,
    \item\label{item:Blow-up:H} an $a$-arrangeable $n$-vertex graph $H$
      with maximum degree $\Delta(H)\le \sqrt{n}/\log n$, together with a
      $\kappa$-balanced $R$-partition $V(H)=X_1\dcup\dots\dcup X_r$,
    \item a graph $G$ with a corresponding
      $(\eps,\delta)$-super-regular $R$-partition $V(G)=V_1\dcup\dots\dcup
      V_r$ with $|V_i|=|X_i|=:n_i$ for every $i\in[r]$,
    \item for every $i\in [r]$ a set $S_i\subseteq X_i$ of at most
      $|S_i|\le \alpha n_i$ \emph{image restricted vertices}, such that
      $|N_H(S_i)\cap X_j|\le \alpha n_j$ for all $ij\in
      E(R)$,
    \item\label{item:Blow-up:ir} and for every $i\in[r]$ a family
      $\cI_i=\{I_{i,1},\dots,I_{i,C}\}\subseteq 2^{V_i}$ of permissible \emph{image
      restrictions}, of size at least $|I_{i,j}|\ge c n_i$ each, together
      with a mapping $I\colon S_i\to \cI_i$, which assigns a permissible
      image restriction to each image restricted vertex.
  \end{enumerate}
  Then there exists an embedding $\phi\colon V(H) \to V(G)$ such that
  $\phi(X_i)=V_i$ and $\phi(x) \in I(x)$ for every $i\in[r]$ and every
  $x\in S_i$.
\end{thm}

As we shall show, the upper bound on the maximum degree of~$H$ in
Theorem~\ref{thm:Blow-up:arr:full} is optimal up to the $\log$-factor (see Section~\ref{sec:opt}).
However, if we require additionally that every $(a+1)$-tuple of~$G$ has a big
common neighbourhood then this degree bound can be relaxed to $o(n/\log n)$. 

\begin{thm}[Arrangeable Blow-up Lemma, extended version] 
  \label{thm:Blow-up:arr:ext} 
  Let $a,\Delta_R,\kappa \in \N$ and $\iota,\delta>0$ be given. Then there
  exist $\eps,\xi>0$ such that for every $r$ there is $n_0\in \N$ such that
  the following holds for every $n\ge n_0$.

  Assume that we are given a graph $R$ of order $r$ with
  $\Delta(R)<\Delta_R$, an $a$-arrangeable $n$-vertex graph with $\Delta(H)
  \le \xi n/\log n$, together with a $\kappa$-balanced $R$-partition, and
  a graph~$G$ with a corresponding $(\eps,\delta)$-super-regular $R$-partition
  $V=V_1\dcup\dots\dcup V_r$. Assume that in addition for every $i\in[r]$
  every tuple $(u_1,\dots,u_{a+1})\subseteq V\setminus V_i$ of vertices
  satisfies $|\bigcap_{j\in[a+1]}N_G(u_j) \cap V_i|\ge \iota |V_i|$.
  Then there exists an embedding of $H$ into $G$.
\end{thm}

Again, the degree bound of $\xi n/\log n$ for $H$ in Theorem~\ref{thm:Blow-up:arr:ext} is optimal up to the constant factor.
The same degree bound can be obtained if we do require~$H$ only to be an almost spanning subgraph, even if the additional condition on $(a+1)$-tuples from Theorem~\ref{thm:Blow-up:arr:ext} is dropped again.

\begin{thm}[Arrangeable Blow-up Lemma, almost spanning
  version] \label{thm:Blow-up:arr:almost-span} Let $\mu>0$ and assume that we have exactly the
  same setup as in Theorem~\ref{thm:Blow-up:arr:full}, but with
  $\Delta(H)\le \xi n/\log n$ instead of the maximum degree bound given
  in~\eqref{item:Blow-up:H}, where~$\xi$ is sufficiently small compared to
  all other constants.  Fix an $a$-arrangeable ordering of~$H$,
  let~$X_i'$
  be the first $(1-\mu)n_i$ vertices of $X_i$ in this ordering, and set
  $H':=H[X_1'\cup\dots\cup X_r']$.

  Then there exists an embedding $\phi\colon V(H') \to V(G)$ such that
  $\phi(X_i')\subseteq V_i$ and $\phi(x) \in I(x)$ for every $i\in[r]$ and
  every $x\in S_i\cap X_i'$.
\end{thm} 

Let us point out that one additional essential difference between these three versions 
of the Blow-up Lemma and 
Theorem~\ref{thm:Blow-up:KSS} concerns the order of the quantifiers: the
regularity~$\eps$ that we require only depends on the maximum degree
$\Delta_R$ of the reduced graph~$R$, but \emph{not} on the number of the vertices in $R$. 
Sometimes this is useful in applications. Clearly, we can reformulate our theorems to match the original order of
quantifiers of Theorem~\ref{thm:Blow-up:KSS}; the lower bound on $n_0$ can
be omitted in this case.

\paragraph{Applications.}

To demonstrate the usefulness of these extensions of the Blow-up Lemma, 
we consider two
example applications that can now be derived in a relatively straightforward manner.
At the end of this section we are going to mention a few further 
applications that are more difficult and will be proven in separate papers. 

\smallskip

Our first application concerns $F$-factors in graphs of high minimum
degree.  This is a topic which is well investigated for graphs~$F$ of
\emph{constant} size. For a graph~$F$ on~$f$ vertices, an \emph{$F$-factor}
in a graph~$G$ is a collection of vertex disjoint copies of~$F$ in~$G$ such
that all but at most $f-1$ vertices of~$G$ are covered by these copies of $F$.

A classical theorem by Hajnal and Szemer\'edi~\cite{HajSze70} states that
each $n$-vertex graph~$G$ with minimum degree $\delta(G)\ge\frac{r-1}{r}n$
has a $K_r$-factor. 
Alon and Yuster~\cite{AloYus} considered arbitrary graphs $F$ and 
showed that, if $r$ denotes the chromatic number of $F$, every sufficiently large 
graph~$G$ with minimum degree $\delta(G)\ge(\frac{r-1}{r}+\gamma)n$
contains an $F$-factor. This was improved upon by Koml\'os,
S\'ark\"ozy, and Szemer\'edi~\cite{KSS_alon}, who replaced the linear term~$\gamma n$ in the degree bound by a constant $C=C(F)$; and by K\"uhn
and Osthus~\cite{KueOst_packings}, who, inspired by a result of Koml\'os~\cite{Kom00},
determined the precise minimum degree threshold for every constant size~$F$
up to a constant.

In contrast to the previous results we consider graphs~$F$ whose size may
grow with the number of vertices~$n$ of the host graph~$G$. More precisely, we allow
graphs~$F$ of size linear in~$n$. To prove this result, we use
Theorem~\ref{thm:Blow-up:arr:full} (see Section~\ref{sec:Applications}) and
hence we require that~$F$ is $a$-arrangeable and has maximum degree at
most $\sqrt{n}/\log n$.

\begin{thm}\label{thm:GrowingHfactors}
  For every $a,r$ and $\gamma>0$ there exist $n_0$ and $\xi>0$ such that
  the following is true. Let $G$ be any graph on $n\ge n_0$ vertices with
  $\delta(G)\ge (\tfrac{r-1}{r}+\gamma)n$ and let $F$ be an $a$-arrangeable
  $r$-chromatic graph with at most $\xi n$ vertices and with maximum degree
  $\Delta(F)\le \sqrt{n}/\log n$. Then $G$ contains an $F$-factor.
\end{thm}

Our second application is a universality result for random graphs $\Gnp$
with constant~$p$ (that is, a graph on vertex set~$[n]$ for which every
$e\in\binom{[n]}{2}$ is inserted as an edge independently with
probability~$p$).  A graph~$G$ is called \emph{universal} for a
family~$\mathcal{H}$ of graphs if~$G$ contains a copy of each graph
in~$\mathcal{H}$ as a subgraph. 
For instance, graphs that are universal for the family of forests, of
planar graphs and of bounded degree graphs have been investigated
(see~\cite{AloCap} and the references therein).

Here we consider the class
\[\mathcal{H}_{n,a,\xi} := \{H: \text{$|H|=n$, $H$ is $a$-arrangeable, $\Delta(H)\le \xi n/\log n$}\}\]
of arrangeable graphs whose maximum degree is allowed to grow with~$n$.
Using Theorem~\ref{thm:Blow-up:arr:ext}, we show that with high probability
$\Gnp$ contains a copy of each graph in $\mathcal{H}_{n,a,\xi}$ (see Section~\ref{sec:Applications}).
Universality problems for bounded degree graphs in
(subgraphs of) random graphs with constant~$p$ were also considered in~\cite{HuaLeeSud}. Another result for subgraphs of potentially growing degree and $p$ tending to 0 can be found in~\cite{Rio00}. Theorem~2.1 of~\cite{Rio00} implies that any $a$-arrangeable graph of maximum degree $o(n^{1/4})$ can be embedded into $\Gnp$ with $p>0$ constant with high probability.

\begin{thm}\label{thm:GnpUniversal}
For all constants $a$, $p>0$ there exists $\xi>0$ such that $\Gnp$ is universal for 
$\mathcal{H}_{n,a,\xi}$
with high probability.
\end{thm}

In addition, we use Theorem~\ref{thm:Blow-up:arr:full} in~\cite{ArrBolKom}
to establish an analogue of the Bandwidth Theorem from~\cite{BoeSchTar09} for
arrangeable graphs. More precisely, we prove the following result.

\begin{thm}[Arrangeable Bandwidth Theorem~\cite{ArrBolKom}]
\label{thm:BolKom:arr}
For all $r,a\in\N$ and $\gamma>0$, there exist constants $\beta>0$ and $n_0\in\N$ such that for every $n\ge n_0$ the following holds. 
If $H$ is an $r$-chromatic, $a$-arrangeable graph on $n$ vertices with $\Delta(H)\le \sqrt{n}/\log n$ and bandwidth at most $\beta n$ and if $G$ is a graph on $n$ vertices with minimum degree $\delta(G)\ge \left(\tfrac{r-1}r+\gamma\right)n$, then there exists an embedding of $H$ into $G$.
\end{thm}

As we also show there, this implies for example that every graph~$G$ with minimum
degree at least $(\frac34+\gamma)n$ contains \emph{almost every} planar graph~$H$
on~$n$ vertices, provided that $\gamma>0$. In addition it implies that
almost every planar graph~$H$ has Ramsey number $R(H)\le12|H|$.

Finally, another application of Theorem~\ref{thm:Blow-up:arr:full} appears
in~\cite{ASW12}.  In that paper Allen, Skokan, and W\"urfl prove the
following result, closing a gap left in the analysis of large planar
subgraphs of dense graphs by K\"uhn, Osthus, and Taraz~\cite{KueOstTar} and
K\"uhn and Osthus~\cite{KueOst_triang}.

\begin{thm}[Allen, Skokan, W\"urfl~\cite{ASW12}]
\label{thm:DensePlanarSubgraphs}
  For every $\gamma\in (0,1/2)$ there exists $n_\gamma$ such that every
  graph on $n\ge n_\gamma$ vertices with minimum degree at least $\gamma n$
  contains a planar subgraph with $2n-4k$ edges, where $k$ is the unique
  integer such that $k\le 1/(2\gamma)<k+1$.
\end{thm}

\paragraph{Methods.}

To prove the full version of our Arrangeable Blow-up Lemma
(Theorem~\ref{thm:Blow-up:arr:full}), we proceed in two steps. Firstly, we
use a random greedy algorithm to embed an almost spanning
subgraph~$H'$ of the target graph~$H$ into the host graph~$G$ (proving
Theorem~\ref{thm:Blow-up:arr:almost-span} along the way). Secondly, we
complete the embedding by finding matchings in suitable auxiliary graphs
which concern the remaining vertices in $V(H)\setminus V(H')$ and the
unused vertices~$V^\Free$ of~$G$. The first step uses an approach similar to the one
of Koml\'os, S\'ark\"ozy, and Szemer\'edi~\cite{KSS97}. The second step
utilises ideas from R\"odl and Ruci\'nski's~\cite{RodlRuci99}.  Let us
briefly comment on the similarities and differences.

The use of a random greedy algorithm to prove the Blow-up Lemma appears
in~\cite{KSS97}. The idea is intuitive and simple: Order the vertices of
the target graph~$H'$ arbitrarily and consecutively embed them into the
host graph~$G$, in each step choosing a random image vertex $\phi(x)$ in
the set $A(x)$ of those vertices which are still possible as images for the
vertex~$x$ of~$H'$ we are currently embedding. If for some unembedded
vertex~$x$ the set~$A(x)$ gets too small, then call~$x$ critical and embed it
immediately, but still randomly in~$A(x)$.
Our random greedy algorithm proceeds similarly, with one main difference.
We cannot use an arbitrary order of the vertices of~$H'$, but have to use
one which respects the arrangeability bound. Consequently, we also cannot
embed critical vertices immediately -- each vertex has to be embedded when
it is its turn according to the given order. So we need a different
strategy for dealing with critical vertices. We solve this problem by
reserving a linear sized set of special vertices in~$G$ for the embedding of
critical vertices, which are very few.

The second step is more intricate. Similarly to the approach
in~\cite{RodlRuci99} we construct for each cluster $V_i$ an auxiliary
bipartite graph~$F_i$ with the classes $X_i\setminus V(H')$ and $V_i\cap V^\Free$
and an edge between $x\in V(H)$ and $v\in V(G)$ whenever embedding $x$
into $v$ is a permissible extension of the partial embedding from the first step.
Moreover, we guarantee that $V(H)\setminus V(H')$ is a stable
set. Then, clearly, if each~$F_i$ has a perfect matching, there is an
embedding of~$H$ into~$G$. So the question remains how to show that the
auxiliary graphs have perfect matchings. R\"odl and Ruci\'nski approach
this by showing that their auxiliary graphs are super-regular. We would
like to use a similar strategy, but there are two main
difficulties. Firstly, because the degrees in our auxiliary graphs vary
greatly, they cannot be super-regular. Hence we have to appropriately
adjust this notion to our setting, which results in a property that we
call weighted super-regular. Secondly, the proof that our auxiliary graphs
are weighted super-regular now has to proceed quite differently, because we
are dealing with the arrangeable graphs.

\paragraph{Structure.}

This paper is organised as follows. In Section~\ref{sec:notation} we
provide notation and some tools. In Section~\ref{sec:almost-spanning} we
show how to embed almost spanning arrangeable graphs, which will prove
Theorem~\ref{thm:Blow-up:arr:almost-span}. In Section~\ref{sec:spanning} we
extend this to become a spanning embedding, proving
Theorem~\ref{thm:Blow-up:arr:full}. At the end of
Section~\ref{sec:spanning}, we also outline how a similar argument gives
Theorem~\ref{thm:Blow-up:arr:ext}.  In Section~\ref{sec:opt} we explain why
the degree bounds in the new versions of the Blow-up Lemma and the requirements for the image
restrictions are essentially best possible.  In
Section~\ref{sec:Applications}, we give the proofs for our applications,
Theorem~\ref{thm:GrowingHfactors} and Theorem~\ref{thm:GnpUniversal}.

\section{Notation and preliminaries}
\label{sec:notation}

All logarithms are to base $e$. For a graph~$G$ we write $V(G)$ for its vertex set, $E(G)$ for its edge set
and denote the number of its vertices by $|G|$, its \emph{maximum degree}
by $\Delta(G)$ and its \emph{minimum degree} by $\delta(G)$. Let $u,v\in
V(G)$ and $U,W\subset V(G)$.  The \emph{neighbourhood} of~$u$ in~$G$ is
denoted by $N_G(u)$, the neighbourhood of~$u$ in the set $U$ by
$N_G(u,U):=N_G(u)\cap U$.  Similarly $N_G(U)=\bigcup_{x\in U} N_G(x)$ and
$N_G(U,W):=N_G(U)\cap W$.  The \emph{co-degree} of~$u$ and~$v$ is
$\deg_G(u,v)=|N_G(u)\cap N_G(v)|$.  We often omit the subscript~$G$.

For easier reading, we will often use $x$, $y$ or $z$ for vertices in
the graph~$H$ that we are embedding, and $u$, $v$, $w$ for vertices of the
host graph $G$.

We shall also use  the following version of the Hajnal-Szemer\'edi Theorem~\cite{HajSze70}.

\begin{thm} \label{thm:HajnalSzemeredi}
Every graph $G$ on $n$ vertices and maximum degree $\Delta(G)$ can be partitioned into $\Delta(G)+1$ stable sets of size $\lfloor n/(\Delta(G)+1)\rfloor$ or $\lceil n/(\Delta(G)+1)\rceil$ each.
\end{thm}

\subsection{Arrangeability}

Let~$H$ be a graph and $(x_1,x_2,\dots,x_n)$ be an $a$-arrangeable ordering
of its vertices. We write $x_i \prec x_j$ if and only if $i<j$ and say that
$x_i$ is \emph{left} of~$x_j$ and $x_j$ is \emph{right} of $x_i$. We write
$N^-(x)\deff \{y\in N_H(x) : y \prec x\}$ and $N^+(x)\deff \{y\in N_H(x): x
\prec y\}$ and call these the set of \emph{predecessors} or the set of
\emph{successors} of $x$ respectively. Predecessors and successors of
vertex sets and in vertex sets are defined accordingly. Then $|N^+(x)| \le
\Delta(H)$ for all $x\in V(H)$ and the definition of arrangeability says
that $N^-\big(N^+(x_i)\big)\cap\{x_1,\dots,x_i\}$ is of size at most~$a$
for each $i\in[n]$. Moreover, it follows that all $x\in V(H)$ satisfy
$|N^-(x)| \le a$ and

\begin{align}  \label{eq:degree-sum}
  e(H) = \sum_{x\in V(H)} |N^+(x)| = \sum_{x\in V(H)} |N^-(x)| \le a n\, .
\end{align}

In the proof of our main theorem, it will turn out to be desirable to have a vertex ordering which is not only arrangeable, but also has the property that its final $\mu n$ vertices form a stable set. More precisely we require the following properties.

\begin{defi}[stable ending] \label{def:StableEnding}
Let $\mu>0$ and let $H = (X_1\dcup\dots\dcup X_r,E)$ be an $r$-partite, $a$-arrangeable graph with partition classes of order $|X_i|=n_i$ with $\sum_{i\in[r]}n_i=n$. Let $(v_1,\dots,v_{n})$ be an $a$-arrangeable ordering of $H$. We say that the ordering has a \emph{stable ending of order} $\mu n$ if $W=\{v_{(1-\mu)n +1},\dots,v_{n}\}$ has the following properties
\begin{enumerate}
\item $|W\cap X_i| = \mu n_i$ for every $i\in[r]$,
\item $H[W]$ is a stable set.
\end{enumerate}
\end{defi}
 
The next lemma shows that an arrangeable order of a graph can be reordered to have a stable ending while only slightly increasing the arrangeability bound.

\begin{lem} \label{lem:arr:StableEnding}
Let $a,\Delta_R,\kappa$ be integers and let $H$ be an $a$-arrangeable graph that has a $\kappa$-balanced $R$-partition with $\Delta(R)<\Delta_R$. Then $H$ has a $(5a^2\kappa\Delta_R)$-arrangeable ordering with stable ending of order $\mu n$, where $\mu = 1/(10a(\kappa\Delta_R)^2)$.
\end{lem}

\begin{proof}
Let $X=X_1\dcup\dots\dcup X_r$ be a $\kappa$-balanced $R$-partition of $H$ with $|X_i|=n_i$. Further let $(x_1,\dots,x_{n})$ be any $a$-arrangeable ordering of $H$. In a first step we will find a stable set $W\subseteq X$ with $|W\cap X_i|=\mu n_i$ for $\mu=1/(10a(\kappa\Delta_R)^2)$. Note that for every $i\in[r]$ a vertex $x\in X_i$ has only neighbours in sets $X_j$ with $ij\in E(R)$. Further $H[X_i\cup \{X_j:ij\in E(R)\}]$ has at most $\kappa\Delta_Rn_i$ vertices and is $a$-arrangeable. Therefore
\begin{align*}
 \sum_{w\in X_i} \deg(w) \leByRef{eq:degree-sum} 2a\kappa\Delta_Rn_i.
\end{align*}
It follows that at least half the vertices $w\in X_i$ have $\deg(w)\le 4a\kappa\Delta_R$. Let $W_i'$ be the set of these vertices and $m'_i$ be their number.

Now we greedily find a stable set $W\subseteq \bigcup_{i\in[r]}W'_i$ as follows. In the beginning we set $W=\emptyset$. Then we iteratively select an $i\in[r]$ with 
\begin{equation}\label{eq:arr:choice}
  |X_i\cap W|/n_i = \min_{j\in [r]} |X_j\cap W|/n_j\,,
\end{equation}
choose an arbitrary vertex $x\in W_i'$, move it to $W$ and delete $x$ from $W_i'$ and $N_H(x)$ from $W_j'$ for all $j\in[r]$. We perform this operation until we have found a stable set $W$ with $|W\cap X_i|=\mu n_i$ for all $i\in[r]$ or we attempt to choose a vertex from an empty set~$W_{i^*}'$.
 
So assume that, at some point, we try to choose a vertex from an empty set $W_{i^*}'$. For each $i\in[r]$ let $m_i$ be the number of vertices chosen from $X_i$ (and moved to~$W$) so far. 
Moreover, let $i\in[r]$ be such that $m_i<\mu n_i$ and consider the last step when a vertex from $X_i$ was chosen. Before this step, $m_i-1$ vertices of~$X_i$ and at most $m_{i^*}$ vertices of~$X_{i^*}$ have been chosen. By~\eqref{eq:arr:choice} we thus have $(m_i-1)/n_i\le m_{i^*}/n_{i^*}$, which implies $m_i\le\kappa m_{i^*}+1$ because $n_i\le\kappa n_{i^*}$.
Hence, since $W'_{i^*}$ became empty, we have
\begin{align*}
 n_{i^*}/2 \le m_{i^*}' &\le m_{i^*} + \sum_{\{i^*,i\}\in E(R)}m_i 4a\kappa\Delta_R\\
 &\le m_{i^*} + (\Delta_R-1)(\kappa m_{i^*}+1) 4a\kappa\Delta_R \le m_{i^*}\, 5a(\kappa\Delta_R)^2\,.
\end{align*}
Thus $m_{i^*}\ge n_{i^*}/(10a(\kappa\Delta_R)^2)$. 
Since we then try to choose from $W'_{i^*}$ we must have $m_{i^*}/n_{i^*}\le m_i/n_i$ by~\eqref{eq:arr:choice}, which implies $m_i\ge n_i/(10a(\kappa\Delta_R)^2)=\mu n_i$. Hence we indeed find a stable set $W$ with $|W\cap X_i|=\mu n_i$ for all $i\in[r]$. 

Given this stable set $W$ we define a new ordering in which these vertices are moved to the end in order to form the stable ending. To make this more precise let $(x_1',\dots,x_n')$ be the vertex ordering obtained from $(x_1,\dots,x_n)$ by moving all vertices of $W$ to the end (in any order). It remains to prove that $(x_1',\dots,x_n')$ is $(5a^2\kappa\Delta_R)$-arrangeable. Let $L_i'=\{x_1',\dots,x_i'\}$ and $R_i'=\{x_{i+1}',\dots,x_n'\}$ be the vertices left and right of $x'_i$ in the new ordering. We have to show that
\begin{align*} 
\big|N\big(N(x_i',R_i'),L_i'\big)\big|\le 5a^2\kappa\Delta_R
\end{align*}
for all $i\in[n]$. This is obvious for the vertices in $W$ because they are now at the end and $W$ is stable. For $x_i\notin W$ let $N'_i=N\big(N(x_i,R'_i),L'_i\big)$ be the set of predecessors of successors of $x_i$ in the new ordering. $N_i$ is defined analogously for the original ordering. Then all vertices in $N_i'\setminus N_i$ are neighbours of predecessors $y$ of $x_i$ in the original ordering with $y\in W$. There are at most $a$ such left-neighbours of $x_i$ and each of these has at most $4a\kappa\Delta_R$ neighbours by definition of $W$. Hence
\[
 |N_i'| \le|N_i| +a\cdot 4a\kappa\Delta_R\le a + 4a^2\kappa\Delta_R \le 5a^2\kappa\Delta_R\,.
\]
\end{proof}

\subsection{Weighted regularity} \label{sec:WeightedRegularity}

In our proof we shall make use of a weighted version of $\eps$-regularity. More precisely, we will have to deal with a bipartite graph whose vertices have very different degrees. The idea is then to give each vertex a weight antiproportional to its degree and then say that the graph is weighted regular if the following holds.

\begin{defi}[Weighted regular pairs] \label{def:weighted:regularity}
Let $\eps>0$ and consider a bipartite graph $G=(A\dcup B,E)$ with a weight function $\omega:A\to [0,1]$. For $A'\subseteq A$ , $B'\subseteq B$ we define the \emph{weighted density}  
\[
   d_\omega(A',B')\deff\frac{\sum_{x\in A'}\omega(x)|N(x,B')|}{|A'|\cdot|B'|}\, .
\]
We say that the pair $(A,B)$ with weight function $\omega$ is \emph{weighted $\eps$-regular} (with respect to $\omega$) if for any $A'\subseteq A$ with $|A'|\ge \eps|A|$ and any $B'\subseteq B$ with $|B'|\ge \eps |B|$ we have
\[
  |d_\omega(A,B) - d_\omega(A',B')| \le \eps\, .
\]
\end{defi}

Many results for $\eps$-regular pairs carry over to weighted $\eps$-regular pairs. For one, subpairs of weighted regular pairs are weighted regular.

\begin{prop}
\label{prop:regular:sub}
Let $G=(A\dcup B,E)$ with weight function $\omega:A\to [0,1]$ be weighted $\eps$-regular. Further let $A'\subseteq A$, $B'\subseteq B$ with $|A'|\ge \gamma|A|$ and $|B'|\ge\gamma |B|$ for some $\gamma\ge\eps$ and set $\eps'\deff\max\{2\eps,\eps/\gamma\}$. Then $(A'\dcup B', E\cap A'\times B')$ is a weighted $\eps'$-regular pair with respect to the restricted weight function $\omega':A'\to [0,1]$, $\omega'(x)=\omega(x)$.
\end{prop}

\begin{proof}
Let $A'\subseteq A$ and $B'\subseteq B$ with $|A'|\ge \gamma |A|$, $|B'|\ge\gamma|B|$ be arbitrary. The definition of weighted $\eps$-regularity implies that $|d_\omega(A,B)-d_\omega(A',B')|\le \eps$. Moreover, $|d_\omega(A,B)-d_\omega(A^*,B^*)|\le \eps$ for all $A^*\subseteq A'$ and $B^*\subseteq B'$ with $|A^*|\ge (\eps/\gamma)|A'|\ge \eps|A|$, $|B^*|\ge (\eps/\gamma)|B'|\ge \eps|B|$ for the same reason. It follows by triangle inequality that $|d_\omega(A',B')-d_\omega(A^*,B^*)|\le 2\eps$. Hence $(A'\dcup B',E\cap A'\times B')$ with weight function $\omega':A'\to[0,1]$ is a weighted $\eps'$-regular pair where $\eps'=\max\{2\eps,\eps/\gamma\}$.
\end{proof}

If most vertices of a bipartite graph have the `right' degree and most pairs have the `right' co-degree then the graph is an $\eps$-regular pair. This remains to be true for weighted regular pairs and weighted degrees and co-degrees.

\begin{defi}[Weighted degree and co-degree] \label{def:weighted:degree}
Let $G=(A\dcup B, E)$ be a bipartite graph and $\omega:A\to [0,1]$. For $x,y\in A$ we define the \emph{weighted degree} of $x$ as $\deg_\omega(x)\deff \omega(x)|N(x,B)|$ and the \emph{weighted co-degree} of $x$ and $y$ as $\deg_\omega(x,y)\deff \omega(x)\omega(y)|N(x,B)\cap N(y,B)|$. 
\end{defi}

A proof of the following lemma can be found in the Appendix.

\begin{lem} \label{lem:reg:weighted:deg-codeg} 
Let $\eps>0$ and $n\ge \eps^{-6}$. Further let $G=(A\dcup B,E)$ be a bipartite graph with $|A|=|B|=n$ and let $\omega: A \to [\eps,1]$ be a weight function for $G$. If
\begin{enumerate}
\item $|\{x\in A: |\deg_\omega(x)-d_\omega(A,B)n| > \eps^{14}n\}| < \eps^{12}n$ \quad and
\item $|\{\{x,y\}\in \binom A2: |\deg_\omega(x,y)-d_\omega(A,B)^2n| \ge \eps^{9}n\}| \le \eps^{6}\binom n2$
\end{enumerate}
then $(A,B)$ is a weighted $3\eps$-regular pair.
\end{lem}

It is well known that a balanced $(\eps,\delta)$-super-regular pair has a
perfect matching if $\delta>2\eps$ (see, e.g., \cite{RodlRuci99}).
Similarly, balanced weighted regular pairs with an appropriate minimum
degree bound have perfect matchings (see the Appendix for a proof).

\begin{lem} \label{lem:reg:weighted:matching} 
Let $\eps>0$ and let $G=(A\dcup B,E)$ with $|A|=|B|=n$ and weight function $\omega: A \to [\sqrt\eps,1]$ be a weighted $\eps$-regular pair. If $\deg(x)>2\sqrt\eps n$ for all $x\in A\cup B$ then $G$ contains a perfect matching.
\end{lem}

\subsection{Chernoff type bounds} \label{sec:ChernoffBounds}

Our proofs will heavily rely on the probabilistic method. In particular we will want to bound random variables that are close to being binomial. By close to we mean that the individual events are not necessarily independent but occur with certain probability even if condition on the outcome of other events. The following two variations on the classical bound by Chernoff make this more precise.

\begin{lem} \label{lem:pseudo:Chernoff}
Let $0\le p_1\le p_2 \le 1$, $0<c \le 1$. Further let $\cA_i$ for $i\in[n]$ be 0-1-random variables and set $\cA:=\sum_{i\in[n]}\cA_i$. If 
\[
   p_1 \le 
   \PP\left[\cA_i=1 \,\left|~\parbox{150pt}{ $\cA_{j}=1$ for all $j \in J$ and\\ $\cA_j=0$ for all $j\in [i-1]\setminus J$}\right. \right] \le p_2
\]
for every $i\in[n]$ and every $J\subseteq[i-1]$ then
\[
  \PP[\cA \le (1-c) p_1n] \le \exp\left(-\frac{c^2}3 p_1n \right)
\]
and
\[
  \PP[\cA \ge (1+c)p_2n] \le \exp\left(-\frac{c^2}3 p_2n \right)\, .
\]
\end{lem}

Similarly we can state a bound on the number of tuples of certain random variables.

\begin{lem} \label{lem:pseudo:Chernoff:tuple}
Let $0<p$ and $a,m,n\in \N$. Further let $\cI\subseteq\mathcal{P}([n])\setminus\{\emptyset\}$ be a collection of $m$ disjoint sets with at most $a$ elements each. For every $i\in [n]$ let $\cA_i$ be a 0-1-random variable. Further assume that for every $I\in\cI$ and every $k\in I$ we have 
\[
  \PP\left[\cA_k=1 \,\left|~\parbox{150pt}{ $\cA_{j}=1$ for all $j \in J$ and\\ $\cA_j=0$ for all $j\in [k-1]\setminus J$}\right. \right] \ge p
\] 
for every $J\subseteq [k-1]$ with $[k-1]\cap I\subseteq J$. Then
\[
\PP\Big[ \big|\{I\in \mathcal{I}: \cA_{i}=1\text{ for all $i\in
        I$}\}\big|\ge \tfrac{1}{2} p^am\Big] \ge
1-2\exp\Big(-\frac1{12} p^am \Big)\, .
\]
\end{lem}

The proofs for both lemmas can be found in the Appendix. The first one is
very close to the proof of the classical Chernoff bound while the second proof builds on the fact that the events $[\cA_i=1 \text{ for all $i\in I$}]$ have probability at least $p^a$ for every $I\in\cI$. In particular, in the special case $a=1$, Lemma~\ref{lem:pseudo:Chernoff} implies Lemma~\ref{lem:pseudo:Chernoff:tuple}.

\section{An almost spanning version of the Blow-up Lemma} \label{sec:almost-spanning}

This section is dedicated to the proof of Theorem~\ref{thm:Blow-up:arr:almost-span} which is a first step towards Theorem~\ref{thm:Blow-up:arr:full}. We give a randomised algorithm for the embedding of an almost spanning subgraph $H'$ into $G$ and show that it is well defined and that it succeeds with positive probability.

This embedding of $H'$ is later extended to the embedding of a spanning subgraph $H$ in Section~\ref{sec:spanning}. Applying the randomised algorithm to $H$ while only embedding $H'$ provides the structural information necessary for the extension of the embedding. It is for this reason that we define a graph $H$ while only embedding a subgraph $H'\subseteq H$ into $G$ in this section.

\begin{remark}
  In the following we shall always assume that each super-regular pair $(V_i,V_j)$ appearing in the proof has density 
  \begin{equation*}
    d(V_i,V_j)=\delta
  \end{equation*}
  \emph{exactly}, and minimum degree
  \begin{equation} \label{eq:min-deg}
    \min\nolimits_{v\in V_i} \deg(v,V_j) \ge \tfrac12\delta|V_j| \,, \quad \min\nolimits_{v\in V_j} \deg(v,V_i) \ge \tfrac12\delta|V_i| \,,
 \end{equation}
 since otherwise we can simply appropriately delete random edges to obtain this situation (while possibly increasing regularity to~$2\eps$).
\end{remark}

\subsection{Constants, constants}

Since there will be plenty of constants involved in the following proofs we
give a short overview first.
\begin{tabbing}
    \quad \= $\Delta_R$: \= the maximum degree of $R$ is \emph{strictly}
      smaller than $\Delta_R$ \\
    \> $r$: \> the number of clusters \\
    \> $a$: \> the arrangeability of $H$ \\
    \> $s$: \> the chromatic number of $H$ \\
    \> $\delta$: \> the density of the pairs $(V_i,V_j)$ in $G$ \\
    \> $\mu$: \> the proportion of $G$ that will be left after embedding
         $H$ \\
    \> $\xi$: \> some constant in the degree-bound of $H$ \\
    \> $\eps$: \> the regularity of the pairs $(V_i,V_j)$ in $G$ \\
    \> $\eps'$: \> the weighted regularity of the auxiliary graphs $F_i(t)$ \\
    \> $\kappa$: \> the maximum quotient between cluster sizes \\
    \> $\gamma$: \> a threshold for moving a vertex into the critical set \\
    \> $\lambda$: \> the fraction of vertices whose predecessors receive a special embedding \\
    \> $\alpha$: \> the fraction of vertices with image restrictions \\
    \> $c$: \> the relative size of the image restrictions \\
    \> $C$: \> the maximum number of image restrictions per cluster \\
\end{tabbing}
Now let $C,a,\Delta_R,\kappa \in\N$ and $\delta,c,\mu>0$ be given. We define the following constants.
\begin{align}
  \gamma &= \frac c2 \frac{\mu}{10} \delta^a\;, \label{eq:consts:gamma} \\
  \lambda &= \frac{1}{25 a}\delta\gamma\;, \label{eq:consts:lambda}\\
  \eps' &= \min\left\{\left(\frac{\lambda \delta^a}{6\cdot 2^{a^2+1}3^a}\right)^2,\left(\frac{7\gamma}{30}\right)^2\right\}\;, \label{eq:consts:eps:p}\\
  \eps &= \min\left\{ \frac1{\Delta_R(1+C)2^{a+1}}\eps', \left(\frac{\eps'}{3}\right)^{36}\right\}\;, \label{eq:consts:eps} \\
  \alpha &=  \frac{\sqrt\eps}6\;. \label{eq:consts:alpha}
\intertext{Furthermore, let $r$ be given. Then we choose}
  \xi &= \frac{8\eps^2}{9\gamma^2\kappa r}\,. \label{eq:consts:xi}
\end{align}
Moreover, we ensure that $n_0$ is big enough to guarantee
\begin{align}
\sqrt{n_0} \ge 48\frac{3^a2^{a^2+1}a \kappa r}{\lambda \delta^a}\;, \quad  
n_0 \ge 60\frac{\kappa r}{\eps^2\delta\mu} \log (12 (n_0)^2) \text{,~~and~~~} 
\log n_0 \ge 36\frac{2^{a^2}a^2\kappa r}{\lambda}
\,. \label{eq:consts:n}
\end{align}
All logarithms are base $e$. In short, the constants used relate as
\[
  0 < \xi \ll \eps \ll \alpha \ll \eps' \ll \lambda \ll \gamma \ll \mu, \delta \le 1\, .
\]
Moreover, $\eps \ll 1/\Delta_R$. Note that it follows from these definitions that $(1+\eps/\delta)^a\le 1+\sqrt\eps/3$ and $(1-\eps/\delta)^a\ge 1-\sqrt\eps/3$ which implies
\begin{align} \label{eq:eps:1}
 \frac{(\delta+\eps)^a}{1+\sqrt{\eps}/3} \le \delta^a \le  \frac{(\delta-\eps)^a}{1-\sqrt{\eps}/3},
 \quad \text{in particular} \quad 
 (\delta-\eps)^a \ge \frac{9}{10} \delta^a.
\end{align}

\subsection{The randomised greedy algorithm} \label{sec:rga:definition}

Let $V(H)=(x_1,\dots,x_n)$ be an $a$-arrangeable ordering of $H$ and let $H'\subseteq H$ be a subgraph induced by $\{x_1,\dots,x_{(1-\mu)n}\}$. In this section we define a \emph{randomised greedy algorithm} (RGA) for the embedding of $V(H')$ into $V(G)$. This algorithm processes the vertices of $H$ vertex by vertex and thereby defines an embedding $\phi$ of $H'$ into $G$. We say that vertex $x_t$ gets embedded in time step $t$ where $t$ runs from 1 to $T=|H'|$. Accordingly $t(x)\in [n]$ is defined to be the time step in which vertex $x$ will be embedded. 

We explain the main ideas before giving an exact definition of the algorithm.

\smallskip

\textbf{Preparing $H$:} Recall that $S_i$ is the set of \emph{image restricted vertices} in $X_i$ and set $S\deff \bigcup S_i$. We define~$L_i^*$ to be the last $\lambda n_i$ vertices in $X_i\setminus N(S)$ in the arrangeable ordering. Moreover, we define $X_i^*\deff N^-(L_i^*) \cup S_i$ and $X^*\deff \bigcup X_i^*$. Those vertices will be called the \emph{important vertices}. The name indicates that they will play a major r\^ole for the spanning embedding. Important vertices shall be treated specially by the embedding algorithm. 
The $a$-arrangeability of $H$ implies that
\begin{align} \label{eq:important:ub}
 |X_i^*| \le a \lambda n_i + \alpha n_i
\end{align}
for all $i\in [r]$.
\smallskip

\textbf{Preparing $G$:} Before we start embedding into $G$ we randomly set
aside $(\mu/10) n_i$ vertices in $V_i$ for each $i\in[r]$. We denote these
sets by $V^\SP_i$ and call them the \emph{special vertices}. All remaining vertices, i.e., $V^\ORD_i\deff V_i\setminus V^\SP_i$ will be called \emph{ordinary vertices}.
As the name suggests our algorithm will try to embed most vertices of $H'$
into the sets $V^\ORD_i$ and only if this fails resort to embedding into
$V^\SP_i$. The idea is that the special vertices will be reserved for the
important vertices and for those vertices in $H'$ whose embedding turns out to be intricate. We define
\[
V^\ORD \deff \bigcup_{i=1}^r V^\ORD_i, \quad
V^\SP \deff \bigcup_{i=1}^r V^\SP_i .
\]
Note that $V^\ORD\dcup V^\SP$ defines a partition of $V(G)$.

\smallskip

\textbf{Candidate sets:}
While our embedding process is running, more and more vertices of $G$ will be used up to accommodate vertices of $H$. For each time step $t\in[n]$ we denote by $V^{\Free}(t)\deff V(G)\setminus \{v\in V(G): \exists t'<t : \phi(x_{t'}) = v\}$ the set of vertices where no vertex of $H$ has been embedded yet. 
Obviously $\phi(x_t) \in V^{\Free}(t)$ for all $t$.

The algorithm will define sets $C_{t,x} \subseteq V(G)$ for $1\le t \le T$, $x\in V(H)$, which we will call 
the \emph{candidate set} for $x$ at time $t$. Analogously 
\[ A_{t,x}\deff C_{t,x}\cap V^{\Free}(t)\]
will be called the \emph{available candidate set} for $x$ at time $t$. Again we distinguish between the \emph{ordinary candidate} set $C^\ORD_{t,x} \deff C_{t,x}\cap V^\ORD$ and the \emph{special candidate set} $C^\SP_{t,x} \deff C_{t,x} \cap V^\SP$ or their respective available version $A^\ORD_{t,x} \deff A_{t,x} \cap V^\ORD$ and $A^\SP_{t,x} \deff A_{t,x}\cap V^\SP$.

Finally we define a set $Q(t)\subseteq V(H)$ and call it the \emph{critical set} at time $t$. $Q(t)$ will contain the vertices whose available candidate set got too small at time $t$ or earlier. 

\smallskip

\noindent
\textbf{Algorithm RGA}

\noindent
\textsc{Initialisation}\\[.5ex]
Randomly select $V^\SP_i\subseteq V_i$ with $|V^\SP_i|=(\mu/10)|V_i|$ for each $i\in [r]$. For $x\in X_i\setminus S_i$ set $C_{1,x}=V_i$ and for $x\in S_i$ set $C_{1,x}=I(x)$. Set $Q(1)=\emptyset$.\\[0.5ex]
Check that for every $i\in[r]$, $v \in V_i$, and every $j\in N_R(i)$ we have 
\begin{align} \label{eq:cond:size:special-neighbours}
\left| \frac{|N_G(v)\cap V_j^\SP|}{|V_j^\SP|} - \frac{|N_G(v)\cap V_j|}{|V_j|} \right| \le \eps\,.
\end{align}
Further check that every $x\in S_i$ has
\begin{align} \label{eq:cond:size:image-restricted}
|C_{1,x}^\SP| = |I(x) \cap V_i^\SP| \ge \tfrac{1}{20}c\mu\, n_i\;.
\end{align}
\textbf{Halt with failure} if any of these does not hold.\\[.5ex]
\noindent\textsc{Embedding Stage}\\[.5ex]
For $t\ge 1$, \textbf{repeat} the following steps.\\[1ex]
\emph{Step~1 -- Embedding $x_t$:} Let $x=x_t$ be the vertex of $H$ to
be embedded at time $t$. 
\noindent
Let $A'_{t,x}$ be the set of vertices~$v\in A_{t,x}$ which
satisfy~\eqref{eq:cond:size-ordinary} and~\eqref{eq:cond:size-special} for all $y\in N^+(x)$:
\begin{align}
(\delta-\eps) |C^\ORD_{t,y}| \le& |N_G(v)\cap C^\ORD_{t,y}| \le (\delta+\eps) |C^\ORD_{t,y}|, \label{eq:cond:size-ordinary}\\ 
(\delta-\eps) |C^\SP_{t,y}| \le& |N_G(v)\cap C^\SP_{t,y}| \le (\delta+\eps) |C^\SP_{t,y}|.\label{eq:cond:size-special}
\end{align}
\noindent
Choose $\phi(x)$ uniformly at random from
\begin{align}\label{eq:Ax}
  A(x)&\deff\begin{cases}
    A^\ORD_{t,x} \cap A'_{t,x} \quad & \text{ if $x \notin X^*$ and $x \notin Q(t)$,} \\
    A^\SP_{t,x} \cap A'_{t,x} & \text{ else.}
  \end{cases}
\intertext{
\emph{Step~2 -- Updating candidate sets:} for each unembedded vertex $y\in V(H)$, set}
\nonumber
C_{t+1,y} &\deff \begin{cases} C_{t,y} \cap N_G(\phi(x)) & \text{ if $y \in N^+(x)$, } \\ 
C_{t,y} &\text{ otherwise.}\end{cases} 
\end{align}
\emph{Step~3 -- Updating critical vertices:} We will call a vertex $y\in X_i$ \emph{critical} if $y\notin X_i^*$ and
\begin{align} \label{eq:cond:queue:in}
|A^\ORD_{t+1,y}| &< \gamma n_i.
\end{align}  
Obtain $Q(t+1)$ by adding to $Q(t)$ all critical vertices that have not been embedded yet. Set $Q_i(t+1)=Q(t+1)\cap X_i$.\\[.5ex]
\textbf{Halt with failure} \textbf{if} there is $i\in [r]$ with 
\begin{align} \label{eq:cond:queue:abort} 
|Q_i(t+1)| &> \eps' n_i\,.
\end{align} 
\textbf{Else, if} there are no more unembedded vertices left in $V(H')$ \textbf{halt with success},\\ \textbf{otherwise} set $t \leftarrow t+1$ and go back to \emph{Step~1}.\\

We have now defined our randomised greedy algorithm for the embedding of an almost spanning subgraph $H'$ into $G$. The rest of this section is to prove that it succeeds with positive probability. This then implies Theorem~\ref{thm:Blow-up:arr:almost-span}.

In order to analyse the RGA we define auxiliary graphs which describe possible embeddings of vertices of $H'$ into $G$. These auxiliary graphs inherit some kind of regularity from $G$ with positive probability. We show that the algorithm terminates successfully whenever this happens. 

In the subsequent Section~\ref{sec:Initialisation} we show that conditions~\eqref{eq:cond:size:special-neighbours} and~\eqref{eq:cond:size:image-restricted} hold with probability at least $5/6$. The \textsc{Initialisation} of the RGA succeeds whenever this happens. Moreover, we prove that the embedding of each vertex is randomly chosen from a set of linear size in Step~1 of the \textsc{Embedding Stage}.

In Section~\ref{sec:AuxiliaryGraph} we define auxiliary graphs and derive that all auxiliary graphs are weighted regular with probability at least $5/6$. We also show that condition~\eqref{eq:cond:queue:abort} never holds if this is the case. Thus the \textsc{Embedding Stage} also terminates successfully with probability at least $5/6$.

We conclude that the whole RGA succeeds with probability at least $2/3$. This implies Theorem~\ref{thm:Blow-up:arr:almost-span}.

\subsection{Initialisation and Step~1} \label{sec:Initialisation}

This section is to prove that the \textsc{Initialisation} of the RGA succeeds with probability at least $5/6$ and that Step~1 of the \textsc{Embedding Stage} always chooses vertices from a non-empty set. 

\begin{lem}\label{lem:RGA:init}
The \textsc{Initialisation} succeeds with probability at least $5/6$, i.e. both condition~\eqref{eq:cond:size:special-neighbours} and~\eqref{eq:cond:size:image-restricted} hold for every $i\in[r]$, $v\in V_i$, $j\in [r]\setminus\{i\}$, and $x\in S_i$ with probability $5/6$.
\end{lem}

\begin{proof}[of Lemma~\ref{lem:RGA:init}]
Fix one $v\in V_i$, $j\in [r]\setminus \{i\}$. Since $V_j^\SP$ is a randomly chosen subset of $V_j$ we have
\[
\EE[ |N_G(v)\cap V_j^\SP|] = |N_G(v)\cap V_j|\frac{|V_j^\SP|}{|V_j|} \ge \frac{\delta}2 n_j \frac{\mu}{10}\,.
\]
It follows from a Chernoff bound (see Theorem~\ref{thm:ChernoffBound} in the Appendix) that 
\[ \PP\Big[ | N_G(v)\cap V_j^\SP| - |N_G(v)\cap V_j| \frac{|V_j^\SP|}{|V_j|} > \eps |V_j^\SP| \Big] \le \exp\left( -\frac{\eps^2}6 \delta n_i \frac{\mu}{10}\right) \leByRef{eq:consts:n} \frac 1{12n^2}\;. \]
Similarly $c|V_i^\SP|\ge c\frac{\mu}{10}n_i$ and
\[
 \PP[ c|V_i^\SP|-|I(x)\cap V_i^\SP| \ge \frac c2 |V_i^\SP|] \le \exp\left(-\frac1{12}c\frac{\mu}{10}n_i\right) \le \frac1{12 n}\;.
\]
A union bound over all $i\in [r]$, $v\in V_i$ and $j\in N_R(i)$ or over all $x\in S_i$ finishes the proof.
\end{proof}

Let us write $\pi(t,x)$ for the number of predecessors of $x$ that already got embedded by time $t$: \[
\pi(t,x)\deff|\{t'<t: \{x,x_{t'}\}\in E(H)\}|.
\]
Obviously $\pi(t,x)\le a$ by the definition of arrangeability.

\begin{lem} \label{lem:RGA:candidate-sets}
Let $x\in X_i\setminus S_i$ and $t\le T$ be arbitrary. Then 
\begin{align*}
	(1-\mu/10)(\delta-\eps)^{\pi(t,x)}n_i &\le|C^\ORD_{t,x}| \le
	(1-\mu/10)(\delta+\eps)^{\pi(t,x)} n_i\, , \\
	(\mu/10)(\delta-\eps)^{\pi(t,x)} n_i &\le |C^\SP_{t,x}| \le
	(\mu/10)(\delta+\eps)^{\pi(t,x)} n_i\,.
\end{align*}
If $x\in S_i$, $t\le T$ then 
\[ \frac{9}{10}\gamma n_i \le |C^\SP_{t,x}|\,. \]
\end{lem}

\begin{proof}
The \textsc{Initialisation} of the RGA defines the candidate sets such that $|C^\ORD_{1,x}| = (1-\mu/10)n_i$ and $|C^\SP_{1,x}| = (\mu/10)n_i$ for every $x\in X_i\setminus S_i$. In the \textsc{Embedding Stage} conditions~\eqref{eq:cond:size-ordinary} and \eqref{eq:cond:size-special} guarantee that $C^\ORD_{t,x}$ and $C^\SP_{t,x}$ respectively shrink by a factor of $(\delta \pm \eps)$ whenever a vertex in $N^-(x)$ is embedded.

If $x\in S_i$ we still have $|C_{1,x}^\SP|\ge  (c\mu/20)n_i$ by \eqref{eq:cond:size:image-restricted}. The statement follows as conditions~\eqref{eq:cond:size-ordinary} and \eqref{eq:cond:size-special} again guarantee that $C_{t,x}^\SP$ shrinks at most by a factor of $(\delta-\eps)^a$. Moreover, $\tfrac{1}{20}c\mu(\delta-\eps)^a\ge \tfrac{9}{10}\gamma$ by~\eqref{eq:eps:1} and the definition of $\gamma$.
\end{proof}

We now argue that $\phi(x)$ is chosen from a non-empty set at the end of Step~1 in the \textsc{Embedding Stage}. In fact, we will show that $\phi(x)$ is chosen from a set of size linear in $n_i$.

\begin{lem} \label{lem:RGA:choices}
For any vertex $x\in X_i$ that gets embedded in the \textsc{Embedding Stage} $\phi(x)$ is chosen randomly from a set $A(x)$ of size at least $(\gamma/2) n_i$.\\
Moreover, if $x$ gets embedded into $V_i^\SP$
\[
  |X_i^*| + |Q_i(t(x))| + |A_{t(x),x}^\SP\setminus A(x)| \le \frac{\delta}{18} |C_{t(x),x}^\SP|\; .
\] 
If the RGA completes the \textsc{Embedding Stage} successfully but $x\in X_i$ does not get embedded in the \textsc{Embedding Stage} we have
\[
   |A_{T,x}^\SP| \ge \frac{7\gamma}{10} n_i\;.
\]
\end{lem}

\begin{proof}
We claim that any $x\in X_i$ that gets embedded into $V_i^\sigma$ during the \textsc{Embedding Stage} has
\begin{align} \label{eq:RGA:choices:cl}
|A_{t(x),x}^\sigma| \ge \frac{7\gamma}{10}n_i\;.
\end{align}
We will establish equation~\eqref{eq:RGA:choices:cl} at the end of this proof. 

In order to show the first statement of the lemma we now bound $|A_{t(x),x}^\sigma\setminus A(x)|$, i.e., we determine the number of vertices that potentially violate conditions~\eqref{eq:cond:size-ordinary} or~\eqref{eq:cond:size-special}. As $H$ is $a$-arrangeable, the vertices $y \in N^+(x)$ share at most $2^a$ distinct ordinary candidate sets $C^\ORD_{t(x),y}$ in each $V_j$. The number of special candidate sets $C^\SP_{t(x),y}$ in each $V_j$ might be larger by a factor of $C$ as they arise from the intersection with at most $C$ sets $I_{j,k}$ (with $k\in [C]$) which are the image restrictions. Moreover, there are less than $\Delta_R$ many sets $V_j$ with $j\in N_R(i)$ bounding the total number of candidate sets we have to care for by $\Delta_R(1+C)2^a$.

As we embed $x$ into an $\eps$-regular pair there are at most $2\eps n_i$ vertices $v\in A^\sigma_{t(x),x}$ for each $C^\ORD_{t(x),y}$ that violate~\eqref{eq:cond:size-ordinary} (and the same number for each $C^\SP_{t(x),y}$ that violate~\eqref{eq:cond:size-special}) with $y\in N^+(x)$. Hence
\begin{align} \label{eq:RGA:excludes}
 |A_{t(x),x}^\sigma\setminus A(x)| \le \Delta_R(1+C)2^{a+1}\eps n_i
\end{align}
if $x$ gets embedded into $V_i^\sigma$. Now $\Delta_R(1+C)2^{a+1}\eps n_i \le \gamma/5 n_i$ by~\eqref{eq:consts:eps} and 
\[
  |A(x)| = |A_{t(x),x}^\sigma| - |A_{t(x),x}^\sigma\setminus A(x)| \ge (\gamma/2)n_i
\]
follows.

\medskip

Next we show the second statement of the lemma. If $x\in X_i$ gets embedded into $V_i^\SP$ in the \textsc{Embedding Stage} we conclude 
\begin{align*}
  |X_i^*| + |Q_i(t(x))| + |A_{t(x),x}^\SP\setminus A(x)| &\le \left(a\lambda + \alpha\right)n_i+ \eps' n_i + \Delta_R(1+C)2^{a+1}\eps n_i \\
& \leBy{\eqref{eq:consts:lambda},\eqref{eq:consts:eps}} \left(\tfrac1{25}\delta\gamma + \alpha + 2\eps'\right)n_i \leByRef{eq:consts:alpha} \tfrac{1}{20}\delta\gamma\,n_i \\
&\le \tfrac{\delta}{18} |C_{t(x),x}^\SP|
\end{align*}
where the first inequality is due to \eqref{eq:important:ub},~\eqref{eq:cond:queue:abort}, and~\eqref{eq:RGA:excludes} and the last inequality is due to $|C_{t(x),x}^\SP|\ge \tfrac{9}{10}\gamma n_i$ by Lemma~\ref{lem:RGA:candidate-sets}.

\medskip

We now return to Equation~\eqref{eq:RGA:choices:cl}. In order to prove it we distinguish between the two cases of~\eqref{eq:Ax} in \emph{Step 1} of the \textsc{Embedding Stage}. If $x\notin X^*$ has never entered the critical set, it is embedded into $A^\ORD_{t(x),x}$ and $|A^\ORD_{t(x),x}| \ge (7\gamma/10)n_i$ holds by condition~\eqref{eq:cond:queue:in}. Else $x$ gets embedded into $A^\SP_{t(x),x}$. As only vertices from $Q_i(t(x))$ or $X_i^*$ have been embedded into $V_i^\SP$ so far, we can bound $|A^\SP_{t(x),x}|$ by
\begin{align*}
|A^\SP_{t(x),x}| &\ge |C^\SP_{t(x),x}| - |Q_i(t(x))| - |X_i^*|\\
 &\geByRef{eq:important:ub} \frac{9\gamma}{10}n_i - \eps' n_i - (a\lambda+\alpha)n_i \ge \frac{7\gamma}{10} n_i
\end{align*}
where the second inequality is due to Lemma~\ref{lem:RGA:candidate-sets} and the third inequality is due to our choice of constants. In any case we have $|A_{t(x),x}^\sigma|\ge \frac{7\gamma}{10}n_i$ if $x$ gets embedded into $V_i^\sigma$ (with $\sigma\in\{\ORD,\SP\}$) in Step~1 of the \textsc{Embedding Stage}.

\medskip

If the RGA completes the \textsc{Embedding Stage} successfully but $x\in X_i$ does not get embedded during the \textsc{Embedding Stage} the analogous argument gives
\begin{align*}
  |A^\SP_{T,x}| &\ge |C^\SP_{T,x}| - |Q_i(T)| - |X_i^*| \ge \frac{7\gamma}{10} n_i\; .
\end{align*}

\end{proof}

\subsection{The auxiliary graph} \label{sec:AuxiliaryGraph}

We run the RGA as described above. In order to analyse it, we define \emph{auxiliary graphs} 
$F_i(t)$ which monitor at every time step $t$ whether a vertex $v\in V(G)$ is still contained in the candidate set of a vertex $x\in V(H)$. Let $F_i(t) \deff (X_i\dcup V_i,E(F_i(t)))$ where $xv \in E(F_i(t))$ if and only if $v\in C_{t,x}$. We stress that we use the candidate sets $C_{t,x}$ and not the set of available candidates $A_{t,x}$. This is well defined as $C_{t,x}\subseteq V_i$ for every $x \in X_i$ and every~$t$. Note that $F_i(t)$ is a balanced bipartite graph. By Lemma~\ref{lem:RGA:candidate-sets} we have
\begin{align}\label{eq:aux:min-max-deg}
   (\delta-\eps)^{\pi(t,x)} n_i \le \deg_{F_i(t)}(x) \le (\delta+\eps)^{\pi(t,x)} n_i
\end{align}
for every $x\in X_i\setminus S_i$, i.e., the degree of $x$ in $F_i(t)$ strongly depends on the number $\pi(t,x)$ of embedded predecessors. 
The main goal of this section is proving, however, that if we take this into account and weight the auxiliary graphs accordingly, then they are with high probability weighted regular (see Lemma~\ref{lem:aux:reg}). It will turn out that the RGA succeeds if this is the case (see Lemma~\ref{lem:reg-small-queue}).

More precisely, for $F_i(t)$ we shall use the weight function $\omega_t\colon X_i\to[0,1]$ with
\begin{equation} \label{eq:weighted:def} 
     \omega_t(x):=\delta^{a-\pi(t,x)}\,.
\end{equation}
Observe that the weight function depends on~$t$. For nicer notation, we write $\deg_{\omega,t}(x)\deff\deg_{\omega(t)}(x)=\omega_t(x)|N_{F_i(t)}(x)|$ for $x\in X_i$ and $d_{\omega,t}(X,Y)\deff d_{\omega(t)}(X,Y)$ for $X\subseteq X_i$ and $Y\subseteq V_i$.
By~\eqref{eq:aux:min-max-deg} we have
\begin{align}
 \deg_{\omega,t}(x) &\ge \delta^{a-\pi(t,x)} (\delta-\eps)^{\pi(t,x)} n_i \geByRef{eq:eps:1} (1-\sqrt\eps/3)\delta^a n_i\,,
\label{eq:weighted:lower-deg} \\
 \deg_{\omega,t}(x) &\le \delta^{a-\pi(t,x')} (\delta+\eps)^{\pi(t,x')} n_i \leByRef{eq:eps:1} (1+\sqrt\eps/3)\delta^a n_i
 \label{eq:weighted:upper-deg}
\end{align}
for every $x\in X_i\setminus S_i$ and $t$. Thus for every $i\in[r]$ and $t\le T$ the auxiliary graph $F_i(t)$ satisfies
\begin{align} \label{eq:aux:reg:weighted:density}
(1-\sqrt\eps/2)\delta^a \leByRef{eq:consts:alpha} (1-\alpha)(1-\sqrt\eps/3)\delta^a \le d_{\omega,t}(X_i,V_i) \le (1+\sqrt{\eps}/2)\delta^a\, .
\end{align}

Let $\cR_i(t)$ denote the event that $F_i(t)$ is weighted $\eps'$-regular for $\eps'$ as in~\eqref{eq:consts:eps:p}. Further let $\cR_i$ be the event that $\cR_i(t)$ for all $t\le T$.

\begin{lem} \label{lem:aux:reg}
We run the RGA in the setting of Theorem~\ref{thm:Blow-up:arr:almost-span}. 
Then $\cR_i$ holds for all $i\in [r]$ with probability at least $5/6$.
\end{lem}
We will use Lemma~\ref{lem:reg:weighted:deg-codeg} and weighted degrees and co-degrees to prove Lemma~\ref{lem:aux:reg}.

\begin{proof}[of Lemma~\ref{lem:aux:reg}] 
This proof checks the conditions of Lemma~\ref{lem:reg:weighted:deg-codeg}. Let 
\begin{align*}
 W_i^{(1)}(t) &= \big\{ x\in X_i : |\deg_{\omega,t}(x)-d_{\omega,t}(X_i,V_i) n_i| > \sqrt{\eps} n_i \big\}\,,\\
 W_i^{(2)}(t) &= \Big\{ \{x,y\} \in \binom{X_i}2 : |\deg_{\omega,t}(x,y)-d_{\omega,t}(X_i,V_i)^2 n_i| \ge \sqrt[4]{\eps} n_i \Big\} 
\end{align*}
be the set of vertices and pairs which deviate from the expected (co-)degree. 
Let $W_i^{(1)}:=\bigcup_{t\in[T]} W_i^{(1)}(t)$ and $W_i^{(2)}:=\bigcup_{t\in[T]} W_i^{(2)}(t)$.
We have $\eps'\ge 3\eps^{1/36}$ by~\eqref{eq:consts:eps}, and by Lemma~\ref{lem:reg:weighted:deg-codeg} all auxiliary graphs $F_i(t)$ with $t=1,\dots,T$ are weighted $\eps'$-regular if both
\begin{align}
|W_i^{(1)}| &< \sqrt{\eps} n_i\,,  \label{eq:aux:reg:co-deg:cond:1}\\
|W_i^{(2)}| &\le \sqrt[4]{\eps} \binom{n_i}2\, . \label{eq:aux:reg:co-deg:cond:2} 
\end{align}
Thus $\cR_i$ occurs whenever equations~\eqref{eq:aux:reg:co-deg:cond:1} and~\eqref{eq:aux:reg:co-deg:cond:2} are satisfied. We will prove that this happens for a fixed $i\in[r]$ with probability at least $1-n_i^{-1}$, which together with a union bound over $i\in[r]$ implies the statement of the lemma.

So fix $i\in[r]$. From~\eqref{eq:weighted:lower-deg},~\eqref{eq:weighted:upper-deg} and~\eqref{eq:aux:reg:weighted:density} we deduce that
\begin{align*}
  | \deg_{\omega,t}(x) - d_{\omega,t}(X_i,V_i) n_i | \le \sqrt{\eps} n_i\
\end{align*}
for all $x \in X_i\setminus S_i$ and every $t\le T$. But $|S_i|\le\alpha n_i < \sqrt\eps n_i$ by~\eqref{eq:consts:alpha} and equation~\eqref{eq:aux:reg:co-deg:cond:1} is thus \emph{always} satisfied.

It remains to consider~\eqref{eq:aux:reg:co-deg:cond:2}. To this end let $P_i$ be the set of all pairs $\{y,z\}\in\binom{X_i\setminus S_i}2$ with $N^-(y)\cap N^-(z)=\emptyset$. Observe that
$|\binom{X_i}2 \setminus \binom{X_i\setminus S_i}2| \le \alpha n_i^2 \le\frac{\sqrt\eps}6 n_i^2$ by~\eqref{eq:consts:alpha} and
\begin{equation*}\begin{split}
  \Big|\Big\{\{y,z\}\in \binom{X_i}2 : N^-(y)\cap N^-(z)\neq\emptyset\Big\}\Big| & \le a\Delta(H)n_i 
  \le a \frac{\xi n}{\log n} n_i \\
  & \le  \frac{2a\xi\kappa r}{\log n}\binom{n_i}2 \leByRef{eq:consts:xi} \sqrt\eps \binom{n_i}2\,.
\end{split}\end{equation*}
Hence it suffices to show that
\begin{equation}\label{eq:aux:almost-final-cond}
  \PP\Big[ |W_i^{(2)} \cap P_i| \le \tfrac12\sqrt[4]{\eps} \binom{n_i}2 \Big]
  \ge \PP\Big[ |W_i^{(2)} \cap P_i| \le \tfrac12\sqrt[4]{\eps} |P_i| \Big]
  > 1-n_i^{-1} \,.
\end{equation}
For this we first partition $P_i$ into sets of mutually predecessor disjoint pairs, i.e., $P_i=K_1\dcup\dots\dcup K_\ell$ such that for every $k\in[\ell]$ no vertex of $X_i$ appears in two different pairs in $K_k$, and moreover no two pairs in $K_k$ contain two vertices that have a common predecessor. Theorem~\ref{thm:HajnalSzemeredi} applied to the following graph asserts that there is such a partition with almost equally sized classes~$K_k$: Let $\mathcal{P}$ be the graph on vertex set $P_i$ with edges between exactly those pairs $\{y_1,y_2\}, \{y'_1,y'_2\} \in P_i$ which have either $\{y_1,y_2\}\cap\{y_1',y_2'\}\neq\emptyset$ or $\big(N^-_H(y_1)\cup N^-_H(y_2)\big)\cap \big(N^-_H(y'_{1})\cup N^-_H(y'_2)\big)\neq \emptyset$.
This graph has maximum degree $\Delta(\mathcal{P})< 2a\Delta(H)n_i \le  2a(\xi n/\log n) n_i$. Hence Theorem~\ref{thm:HajnalSzemeredi} gives a partition $K_1\dcup\dots\dcup K_\ell$ of~$P_i$ into stable sets with $|K_k|\ge \lfloor |P_i|/\big(\Delta(\mathcal{P})+1\big)\rfloor \ge \log n/(8a\xi\kappa r)$ for all $k\in[\ell]$, where we used $|P_i|\ge n_i^{2}/4$.

Now fix $k\in[\ell]$ and consider the random variable $K'_k\deff K_k\cap W_i^{(2)}$. Our goal now is to show
\begin{equation}\label{eq:aux:final-cond}
  \PP\Big[ |K'_k| > \tfrac12\sqrt[4]{\eps} |K_k| \Big] \le n_i^{-3} \,,
\end{equation}
as this together with another union bound over $k\in[\ell]$ with $\ell<n_i^2$ implies~\eqref{eq:aux:almost-final-cond}.
We shall first bound the probability that some fixed pair $\{y,z\}\in K_k$
gets moved to $W_i^{(2)}(t)$ (and hence to $K'_k$) at some time $t$.

For a pair $\{y,z\}\in K_k$ and $t\in[T]$ let $\cCo_{t,y,z}$ denote the
event that $| \deg_{\omega,t+1}(y,z) - \deg_{\omega,t}(y,z)| \le \eps n_i$.
Why are we interested in these events?
Obviously $\cCo_{t,y,z}$ holds for all time steps $t$ with $x_t \notin
N^-(y)\dcup N^-(z)$. This is because we have
$\deg_{\omega,t+1}(y,z)=\deg_{\omega,t}(y,z)$ for such~$t$. 
Moreover $|d_{\omega,t'}(X_i,V_i)^2-\delta^{2a}|\le 2\sqrt{\eps}\delta^{2a}$ by~\eqref{eq:aux:reg:weighted:density}.
Thus the fact that $|N^-(y)\dcup N^-(z)|\le 2a$ and the definition of~$\omega$ from~\eqref{eq:weighted:def} imply the following. 
If $\cCo_{t,y,z}$ holds for all $t\le T$, then 
\begin{align*}
 |\deg_{\omega,t'}(y,z) - d_{\omega,t'}(X_i,V_i)^2n_i| &\le |\deg_{\omega,t'}(y,z)-\delta^{2a}n_i| + |d_{\omega,t'}(X_i,V_i)^2-\delta^{2a}|\,n_i\\
 &\le 2a\eps n_i + 2\sqrt{\eps}\delta^{2a} n_i \leByRef{eq:consts:eps} \sqrt[4]{\eps} n_i
\end{align*}
for every $t'\le T$. 
In other words, if $\cCo_{t,y,z}$ holds for all $t\le
T$ then $\{y,z\}\not\in K_k\cap W_i^{(2)}$. More precisely, we have
the following.
\begin{fact}\label{fac:aux:Co}
  For the smallest~$t$ with $\{y,z\}\in W_i^{(2)}(t)$ we have that
  $\cCo_{t',y,z}$ holds for all $t'<t$ but \emph{not} for $t'=t$.
\end{fact}

Moreover, if $\cCo_{t',y,z}$ holds for all $t'<t$ then
\[
|\deg_{\omega,t}(y,z) - \deg_{\omega,0}(y,z)| \le
\big(\pi(t,y)+\pi(t,z)\big)\eps n_i\,.
\]
Recall that $\deg_{\omega,t}(y,z)=\delta^{a-\pi(t,y)}\delta^{a-\pi(t,z)}|C_{t,y}\cap
C_{t,z}|$ and in particular (since $y,z\in X_i\setminus S_i$) $\deg_{\omega,0}(y,z)=\delta^{2a}n_i$. Hence
\begin{align}\label{eq:aux:assumingCo}
  |C_{t,y}\cap C_{t,z}|\ge (\delta^{\pi(t,y)+\pi(t,z)}-\eps \delta^{\pi(t,y)+\pi(t,z)-2a}(\pi(t,y)+\pi(t,z)))n_i\geByRef{eq:consts:eps} \eps n_i\,.
\end{align}
We now claim that
\begin{align} \label{eq:aux:reg:prob}
 \PP[\cCo_{t,y,z} \, | \, \cCo_{t',y,z}\text{ for all $t'<t$}] \ge 1-\frac{4\eps}\gamma\;.
\end{align}
This is obvious if $x_t\notin N^-(y)\dcup N^-(z)$. So assume we are about to embed an $x_t\in N^-(y)\dcup N^-(z)$, which happens to be in~$X_j$. Then $\phi(x_t)$ is chosen randomly among at least $(\gamma/2) n_j$ vertices of $V_j$ by Lemma~\ref{lem:RGA:choices}. Out of those at most $2\eps n_j$ vertices $v\in V_j$ have 
\[ \big|\deg(v,C_{t,y}\cap C_{t,z}) - d(V_i,V_j) \cdot |C_{t,y}\cap C_{t,z}| \big| > \eps n_i
\]
because $|C_{t,y}\cap C_{t,z}|\ge \eps n_i$ by~\eqref{eq:aux:assumingCo}
and $G[V_i,V_j]$ is $\eps$-regular. For every other choice of $\phi(x_t)=v\in V_j$ we have
\begin{align*}
  \big|\deg_{\omega,t+1}(y,z) &- \deg_{\omega,t}(y,z)\big| \\
& = \big| \omega_{t+1}(y)\omega_{t+1}(z) \deg(v,C_{t,y}\cap C_{t,z}) - \omega_t(y)\omega_t(z) \cdot |C_{t,y}\cap C_{t,z}|\big|\\
& = \omega_{t+1}(y)\omega_{t+1}(z) \cdot \big|\deg(v,C_{t,y}\cap C_{t,z}) - \delta\cdot |C_{t,y}\cap C_{t,z}| \big|\\
& = \omega_{t+1}(y)\omega_{t+1}(z) \cdot \big|\deg(v,C_{t,y}\cap C_{t,z}) - d(V_i,V_j)\cdot |C_{t,y}\cap C_{t,z}| \big|\\
& \le \omega_{t+1}(y)\omega_{t+1}(z) \cdot \eps n_i \le \eps n_i\,.
\end{align*}
Thus at most $2\eps n_j$ out of $(\gamma/2)n_j$ choices for $\phi(x_t)$
will result in $\overline{\cCo_{t,y,z}}$, which
implies~\eqref{eq:aux:reg:prob}, as claimed. 

Finally, in order to show concentration, we will apply
Lemma~\ref{lem:pseudo:Chernoff}. For this purpose observe that  by the construction of~$K_k$ for each
time step $t\in[T]$ the embedding of~$x_t$ changes the co-degree of at most
one pair in~$K_k$, which we denote by $\{y_t,z_t\}$ if present. That is, $x_t \in N^-(y_t)\cup N^-(z_t)$.
Now let $T_k\subseteq [T]$ be the set of time steps $t$ with $\{y_t,z_t\}$
in $K_k$, i.e., let $T_k$ be the set of time steps which actually change
the co-degree of a pair in $K_k$. Since $|N^-(y)\cup N^-(z)|\le 2a$ for
every pair $\{y,z\}\in K_k$ we have $|T_k|\le 2a|K_k|$.
 We define the following 0-1-variables $\cA(t)$ for $t\in T_k$: Let $\cA(t)=1$ if and only if
$\cCo_{t',y_t,z_t}$ holds for all $t'\in [t-1]\cap T_k$ but not for $t'=t$.
Fact~\ref{fac:aux:Co} then implies $|K'_k|\le\cA:=\sum_{t\in T_k}\cA(t)$.
Moreover, for any $t'<t$ with $\{y_{t'},z_{t'}\}=\{y_t,z_t\}$ and
$\cA(t')=1$ we have $\cA(t)=0$ by definition.
Hence, for any $t\in T_k$ and $J\subseteq[t]\cap T_k$ we have
\[
  \PP\left[\cA(t)=1\,\Big|\,\parbox{170pt}{$\cA(t')=1$ for all $t'\in J$\\ $\cA(t')=0$ for all $t'\in ([t]\cap T_k)\setminus J$}\right]\le 4\eps/\gamma
\]
by~\eqref{eq:aux:reg:prob}. Now either $|T_k|<16a\eps|K_k|/\gamma$ and thus $\cA < 16a\eps|K_k|/\gamma$ by definition. Or $|T_k|\ge 16a\eps|K_k|/\gamma$ and 
\[
\PP\left[\cA\ge \frac{16a\eps}\gamma |K_k|\right]\le\PP\left[\cA\ge \frac{8\eps}\gamma |T_k|\right]\le\exp\left(-\frac{4\eps}{3\gamma}|T_k|\right)\le n_i^{-3}, 
\]
by Lemma~\ref{lem:pseudo:Chernoff}, where the last inequality follows from 
\[
 \frac{4\eps}{3\gamma}|T_k| \ge \frac{64a\eps^2}{3\gamma^2}|K_k|\ge \frac{8\eps^2\log n}{3\gamma^2\xi\kappa r} \geByRef{eq:consts:xi} 3\log n_i\,.
\]
Since $|K'_k|\le\cA$ and $16a\eps/\gamma <\frac12\sqrt[4]{\eps}$ by~\eqref{eq:consts:eps} we obtain~\eqref{eq:aux:final-cond} as desired.
\end{proof}

We have now established that the auxiliary graph $F_i(t)$ for the embedding of $X_i$ into $V_i$ is weighted regular for all times $t\le T$ with positive probability. The following lemma states that no critical set ever gets large in this case, i.e., if all auxiliary graphs remain weighted regular, then the RGA terminates successfully.

\begin{lem} \label{lem:reg-small-queue}
For every $t\le T$ and $i\in[r]$ we have: $\cR_i(t)$ implies that $|Q_i(t)| \le\eps' n_i$. In particular, $\cR_i$ for all $i\in[r]$ implies that the RGA completes the \textsc{Embedding Stage} successfully.
\end{lem}

\begin{proof}
The idea of the proof is the following. Vertices only become critical because their available candidate set is significantly smaller than the average available candidate set. In other words, the weighted density between the set of critical vertices and $V_i^{\Free}(t)$ deviates significantly from the weighted density of the auxiliary graph. Since the auxiliary graph is weighted regular it follows that there cannot be many critical vertices.

Indeed, assume for contradiction that there is $i\in[r]$ and $t\le T$ with $|Q_i(t)|>\eps' n_i$ and such that $F_i(t)$ is weighted $\eps'$-regular. Let $x\in Q_i(t)$ be an arbitrary critical vertex. Then $x$ is an ordinary vertex and the available (ordinary) candidate set $A_{t,x}^\ORD=C_{t,x}\cap V_i^\ORD\cap V_i^{\Free}(t)$ of $x$ got small, that is, 
\begin{align*}
 |C_{t,x}\cap V_i^\ORD\cap V_i^{\Free}(t)| \lByRef{eq:cond:queue:in} \gamma n_i \leByRef{eq:consts:gamma} \frac{\mu}{20}\delta^a n_i\,.
\end{align*}
In the language of the auxiliary graph this means that
\begin{align*}
\deg_{\omega,t}(x, V_i^\ORD \cap V_i^{\Free}(t)) &= \omega_t(x) |C_{t,x}\cap V_i^\ORD \cap V_i^{\Free}(t)| \le \frac{\mu}{20}\delta^a n_i .
\end{align*}
Moreover $|V_i^\ORD\cap V_i^{\Free}(t)| \ge |V_i^{\Free}(t)|-|V_i^\SP| \ge \tfrac{9}{10} \mu n_i\ge \eps' n_i$. This implies
\begin{align} \label{eq:weighted:upper-irreg}
 d_{\omega,t}( Q_i(t), V_i^\ORD\cap V_i^{\Free}(t) ) \le \frac{\mu/20\, \delta^a n_i}{9/10\, \mu n_i} = \frac1{18} \delta^a\, .
\end{align}
Since~\eqref{eq:aux:reg:weighted:density} and \eqref{eq:weighted:upper-irreg} imply that
\[ d_{\omega,t}(X_i,V_i) - d_{\omega,t}( Q_i(t), V_i^\ORD\cap V_i^{\Free}(t) ) \ge \frac12 \delta^a - \frac1{18} \delta^a > \eps', \]
but $F_i(t)$ is weighted $\eps'$-regular we conclude that $|Q_i(t)|< \eps' n_i$.
\end{proof}

Theorem~\ref{thm:Blow-up:arr:almost-span} is now immediate from the following lemma.

\begin{lem}\label{lem:rga:success}
  If we apply the RGA in the setting of
  Theorem~\ref{thm:Blow-up:arr:almost-span}, then with probability at least
  $2/3$ the event $\cR_i$ holds for all $i\in[r]$ and the RGA finds an
  embedding of $H'$ into $G$ (obeying the $R$-partitions of $H$ and $G$ and
  the image restrictions).
\end{lem}

\begin{proof}[of Lemma~\ref{lem:rga:success}]
Let $C,a,\Delta_R,\kappa$ and $\delta,c,\mu$ be given. Set the constants $\gamma,\eps,\alpha$ as in \eqref{eq:consts:gamma}-\eqref{eq:consts:alpha}. Let $r$ be given and choose $n_0,\xi$ as in~\eqref{eq:consts:xi}-\eqref{eq:consts:n}. Further let $R$ be a graph of order $r$ with $\Delta(R)<\Delta_R$ and let $G,H,H'$ have the required properties. Run the RGA with these settings. The \textsc{Initialisation} succeeds with probability at least $5/6$ by Lemma~\ref{lem:RGA:init}. It follows from Lemma~\ref{lem:aux:reg} that $\cR_i$ occurs for all $i\in [r]$ with probability at least $5/6$. This implies that no critical set $Q_i$ ever violates the bound~\eqref{eq:cond:queue:abort} by Lemma~\ref{lem:reg-small-queue}. Thus the \textsc{Embedding Stage} also succeeds with probability $5/6$. We conclude that the RGA succeeds with probability at least $2/3$. Thus an embedding $\phi$ of $H'=H[X_1'\dcup\dots\dcup X_r']$ into $G$ which maps $X_i'$ into $V_i$ exists. Moreover this embedding guarantees $\phi(x)\in I(x)$ for all $x\in S_i\cap X_i'$ by definition of the algorithm.
\end{proof}

At the end of this section we want to point out that the minimum degree bound for $H$ in Theorem~\ref{thm:Blow-up:arr:almost-span} can be increased even further if we swap the order of the quantifiers. More precisely, for a fixed graph $R$ we may choose $\eps$ such that almost spanning subgraphs of linear maximum degree can be embedded into a corresponding $(\eps,d)$-regular $R$-partition.

\begin{thm} \label{thm:Blow-up:arr:almost-span:lin}
Given a graph $R$ of order $r$ and positive parameters $a,\kappa,\delta,\mu$ there are $\eps,\xi>0$ such that the following holds.
  Assume that we are given 
  \renewcommand{\theenumi}{\alph{enumi}}
  \begin{enumerate}
    \item a graph $G$ with a $\kappa$-balanced
      $(\eps,\delta)$-regular $R$-partition $V(G)=V_1\dcup\dots\dcup
      V_r$ with $|V_i|=:n_i$ and
    \item an $a$-arrangeable graph $H$ with maximum degree $\Delta(H)\le \xi n$ (where $n=\sum n_i$), together with a corresponding $R$-partition $V(H)=X_1\dcup\dots\dcup X_r$ with $|X_i|\le(1-\mu)n_i$.
  \end{enumerate}
  Then there is an embedding $\phi\colon V(H) \to V(G)$ such that
  $\phi(X_i)\subseteq V_i$.
\end{thm}

\begin{proof}[(sketch)]
Theorem~\ref{thm:Blow-up:arr:almost-span:lin} is deduced along the lines of the proof of Theorem~\ref{thm:Blow-up:arr:almost-span}. 

Once more the randomised greedy algorithm from
Section~\ref{sec:rga:definition} is applied. It finds an embedding of $H$
into $G$ if all auxiliary graphs $F_i(t)$ remain weighted regular
throughout the \textsc{Embedding Stage}. This in turn happens if each
auxiliary graph $F_i(t)$ contains few pairs $\{x,y\}\in \binom{X_i}2$ whose
weighted co-degree deviates from the expected value. 

In the setting of Theorem~\ref{thm:Blow-up:arr:almost-span} this is the case with positive probability as has been proven in Lemma~\ref{lem:aux:reg}: Inequality~\eqref{eq:aux:almost-final-cond} states that the number of pairs with incorrect co-degree exceeds the bound of~\eqref{eq:aux:reg:co-deg:cond:2} with probability at most $1-n_i^{-1}$. This particular argument is the only part of the proof of Theorem~\ref{thm:Blow-up:arr:almost-span} that requires the degree bound of $\Delta(H)\le \xi n/\log n$. 
We then used~\eqref{eq:aux:almost-final-cond} and a union bound over
$i\in[r]$ to show that 
all auxiliary graphs $F_i(t)$ remain weighted regular throughout the
\textsc{Embedding Stage} with probability at least $5/6$. Since $r$ can be
large compared to all constants except $n_0$ we need the bound $1-n_i^{-1}$ in~\eqref{eq:aux:almost-final-cond}.

In the setting of Theorem~\ref{thm:Blow-up:arr:almost-span:lin} however it
suffices to replace this bound by a constant. More precisely, since we are allowed to choose $\eps$ depending on the order of $R$ the proof of Lemma~\ref{lem:aux:reg} becomes even simpler: Set $\eps,\xi$ small enough to ensure $8a\eps/\gamma+2a\kappa r\xi\le \sqrt[4]\eps/(6r)$.
Note that inequality~\eqref{eq:aux:reg:prob} then implies that the expected number of pairs $\{y,z\}\in \binom{X_i}2$ with incorrect co-degree is bounded by $2a\frac{4\eps}\gamma \binom{n_i}2+a\Delta(H)n_i \le \frac{\sqrt[4]{\eps}}{6r}\binom{n_i}2$. 

It follows from Markov's inequality and the union bound over all $i\in[r]$ that all auxiliary graphs $F_i(t)$ remain weighted regular throughout the \textsc{Embedding Stage} with probability at least $5/6$.
Choosing $\eps$ sufficiently small we can thus guarantee that the randomised greedy algorithm successfully embeds $H$ into $G$ with positive probability.
\end{proof}

Using the classical approach of Chvatal, R\"odl, Szemer\'edi, and
Trotter~\cite{CRST} Theorem~\ref{thm:Blow-up:arr:almost-span:lin} easily
implies that all $a$-arrangeable graphs have linear Ramsey numbers. This
result has first been proven (using the approach of~\cite{CRST}) by Chen and Schelp~\cite{CheSch93}.

\section{The spanning case} \label{sec:spanning}

In this section we prove our main result, Theorem~\ref{thm:Blow-up:arr:full}. We use the randomised greedy algorithm and its analysis from Section~\ref{sec:almost-spanning} to infer that the almost spanning embedding found in Theorem~\ref{thm:Blow-up:arr:almost-span} can in fact be extended to a spanning embedding. We shortly describe our strategy in Section~\ref{sec:Proof:outline} and establish a minimum degree bound for the auxiliary graphs in Section~\ref{sec:Proof:aux:min-deg} before we give the proof of Theorem~\ref{thm:Blow-up:arr:full} in Section~\ref{sec:Proof:Blow-up:arr:full}. We conclude this section with a sketch of the proof of Theorem~\ref{thm:Blow-up:arr:ext} in Section~\ref{sec:Proof:Blow-up:arr:ext}.

\subsection{Outline of the proof} \label{sec:Proof:outline}

Let $G, H$ satisfy the conditions of Theorem~\ref{thm:Blow-up:arr:full}.
We first use Lemma~\ref{lem:arr:StableEnding} to order the
vertices of~$H$ such that the arrangeability of the resulting order is
bounded and its last $\mu n$ vertices form a stable set~$W$.  We then run
the RGA to embed the almost spanning subgraph $H'=H[X\setminus W]$ into~$G$. The RGA is
successful and the resulting auxiliary graphs $F_i(T)$ are all
weighted regular (that is, $\cR_i$ holds) with probability $2/3$ by
Lemma~\ref{lem:rga:success}.

It remains to extend the embedding of $H'$ to an embedding of $H$.
Since~$W$ is stable it suffices to find for each $i\in[r]$ a bijection
between 
\begin{align*}
  L_i \deff X_i \setminus W
\end{align*}
and $V_i^{\Free}(T)$ which respects the candidate sets, i.e., which maps $x$ into $C_{T,x}$. Such a bijection is given by a perfect matching in $F_i^*\deff F_i(T)[L_i\dcup V_i^{\Free}(T)]$, which is the subgraph of $F_i(T)$ induced by the vertices left after the \textsc{Embedding Phase} of the RGA. 

By Lemma~\ref{lem:reg:weighted:matching} balanced weighted regular pairs
with an appropriate minimum degree bound have perfect matchings.  Now,
$(L_i,V_i^{\Free})$ is a subpair of a weighted regular pair and thus
weighted regular itself by Proposition~\ref{prop:regular:sub}. Hence our main goal is to
establish a minimum degree bound for $(L_i,V_i^{\Free})$. More precisely we
shall explain in Section~\ref{sec:Proof:aux:min-deg} that it easily
follows from the definition of the RGA that vertices in $L_i$ have the
appropriate minimum degree if $\cR_i$ holds.

\begin{prop} \label{prop:aux:left-min-deg} 
Run the RGA in the setting of Theorem~\ref{thm:Blow-up:arr:full} and assume that $\cR_j$ holds for all $j\in [r]$. Then every $x\in L_i$ has
\[ \deg_{F_i(T)}(x,V_i^{\Free}(T)) \ge 3\sqrt{\eps'}n_i.\]
\end{prop}

For vertices in~$V_i^{\Free}$ on the other hand this is not necessarily true. But it holds with sufficiently high probability. This is also proved in Section~\ref{sec:Proof:aux:min-deg}.

\begin{lem}\label{lem:aux:right-min-deg}
Run the RGA in the setting of Theorem~\ref{thm:Blow-up:arr:full} and assume that $\cR_j$ holds for all $j\in[r]$. Then we have
\[ 
\PP\left[\forall i\in [r],~\forall v\in V_i^{\Free}(T): \deg_{F_i(T)}(v,L_i) \ge 3\sqrt{\eps'}n_i \right] \ge \frac23 \, .
\]
\end{lem}

\subsection{Minimum degree bounds for the auxiliary graphs} \label{sec:Proof:aux:min-deg}

In this section we prove Proposition~\ref{prop:aux:left-min-deg} and Lemma~\ref{lem:aux:right-min-deg}. For the former we need to show that vertices $x\in L_i$ have an appropriate minimum degree in $F_i^*$, which is easy.

\begin{proof}[of Proposition~\ref{prop:aux:left-min-deg}]
Since $\cR_j$ holds for all $j\in [r]$ the RGA completed the \textsc{Embedding Stage} successfully by Lemma~\ref{lem:reg-small-queue}. Note that all $x \in L_i$ did not get embedded yet. Thus 
\begin{align*}
    \deg_{F_i(T)}(x,V_i^{\Free}(T)) = |A_{T,x}| \ge |A_{T,x}^\SP| &\ge \tfrac{7}{10}\gamma n_i \geByRef{eq:consts:eps:p} 3\sqrt{\eps'} n_i\,,
\end{align*}
for every $x\in L_i$ by Lemma~\ref{lem:RGA:choices}.
\end{proof} 

Lemma~\ref{lem:aux:right-min-deg} claims that vertices in $V_i^{\Free}(T)$ with positive probability also have a sufficiently large degree in $F_i^*$. We sketch the idea of the proof.

Let $x\in L_i$ and $v\in V_i^{\Free}(T)$ for some $i\in[r]$. Recall that there is an edge $xv \in E(F_i(T))$ if and only if $\phi(N^-(x)) \subseteq N_G(v)$. So we aim at lower-bounding the probability that $\phi(N^-(x))\subseteq N_G(v)$ for many vertices $x\in L_i$. 

Now let $y \in N^-(x)$ be a predecessor of~$x$. Recall that~$y$ is randomly embedded into~$A(y)$, as defined in~\eqref{eq:Ax}. Hence the probability that $y$ is embedded into $N_G(v)$ is $|A(y)\cap N_G(v)|/|A(y)|$. Our goal will now be to show that these fractions are bounded from below by a constant for all predecessors of many vertices~$x\in L_i$, which will then imply Lemma~\ref{lem:aux:right-min-deg}. To motivate this constant lower bound observe that a random subset $A$ of $X_j$ satisfies $|A\cap N_G(v)|/|A|=|N_G(v)\cap V_j|/|V_j|$ in expectation, and the right hand fraction is bounded from below by $\delta/2$ by~\eqref{eq:min-deg}. For this reason we call the vertex $v$ \emph{likely for $y\in X_j$} and say that $\cA_v(y)$ holds, if
\begin{align*}
  \frac{|A(y)\cap N_G(v)|}{|A(y)|} \ge \frac{2}{3}\frac{|V_j \cap N_G(v)|}{|V_j|}\,.
\end{align*}
Hence it will suffice to prove that for every $v\in V_i$ there are many $x\in L_i$ such that $v$ is likely for all $y\in N^-(x)$. 

We will focus on the last $\lambda n_i$ vertices $x$ in $L_i\setminus N(S)$ (i.e., on vertices $x\in L_i^*$) as we have a good control over the embedding of their predecessors (who are in $X^*\setminus S$). Note that there indeed are $\lambda n_i$ vertices in $L_i \setminus N(S)$ as $\mu n_i - \alpha n_i \ge \lambda n_i$. For $i\in[r]$ and $v\in V_i$ we define
\begin{align*}
L_i^*(v)&\deff\{x\in L_i^*: \cA_v(y)\text{ holds for all $y\in N^-(x)$}\}\,.
\end{align*}
Our goal is to show that a positive proportion of the vertices in~$L_i^*$ will be in $L_i^*(v)$. The following lemma makes this more precise.

\begin{lem} \label{lem:aux:rmd:Avy}
We run the RGA in the setting of Theorem~\ref{thm:Blow-up:arr:full} and assume that $\cR_j$ holds for all $j\in[r]$. 
Then
\[ 
  \PP\left[\forall i\in[r], \forall v\in V_i: |L_i^*(v)|\ge 2^{-a^2-1}|L_i^*|\right] \ge \frac56\,.
\]
\end{lem}

Lemma~\ref{lem:aux:rmd:Avy} together with the subsequent lemma will imply Lemma~\ref{lem:aux:right-min-deg}. 

\begin{lem}\label{lem:aux:rmd:pre}
Run the RGA in the setting of Theorem~\ref{thm:Blow-up:arr:full} and assume that $\cR_j$ holds for all $j\in[r]$ and that $|L_i^*(v)|\ge 2^{-a^2-1}|L_i^*|$. Then we have
\[ 
\PP\left[\forall i\in [r],~\forall v\in V_i^{\Free}(T): \deg_{F_i(T)}(v,L_i) \ge 3\sqrt{\eps'}n_i \right] \ge \frac56 \, .
\]
\end{lem}

\begin{proof}[of Lemma~\ref{lem:aux:rmd:pre}]
Let $i\in[r]$ and $v\in V_i$ be arbitrary and assume that the event of Lemma~\ref{lem:aux:rmd:Avy} occurs, this is, assume that we do have $|L_i^*(v)|\ge 2^{-a^2-1}|L_i^*|$. We claim that $v$ almost surely has high degree in $F_i(T)$ in this case.
\begin{claim}
If $|L_i^*(v)|\ge 2^{-a^2-1}|L_i^*|$, then 
\[
  \PP\left[\deg_{F_i(T)}(v,L_i)\ge 3\sqrt{\eps'}n_i\right] \ge 1-\frac{1}{n_i^2}\;.
\]
\end{claim}
This claim, together with a union bound over all $i\in[r]$ and $v\in V_i$, implies that
\[
  \PP\left[\forall i\in[r], \forall v\in V_i: \deg_{F_i(T)}(v,L_i)\ge 3\sqrt{\eps'}n_i\right] \ge \frac 56  
\]
if $|L_i^*(v)|\ge 2^{-a^2-1}|L_i^*|$ for all $i\in[r]$ and all $v\in V_i$. It remains to establish the claim.

\begin{claimproof}[of Claim]
Let~$x\in L_i^*(v)$. Recall that $xv\in E(F_i)$ if and only if $\phi(y)\in N_G(v)$ for all $y\in N^-(x)$. If the events $[\phi(y)\in N_G(v)]$ were independent for all $y \in N^-(L_i^*(v))$ we could apply a Chernoff bound to infer that almost surely a linear number of the vertices $x\in L_i^*(v)$ is such that $[\phi(y)\in N_G(v)]$ for all $y\in N^-(x)$. However, the events might be far from independent: just imagine two vertices $x$, $x'$ sharing a predecessor $y$. We address this issue by partitioning the vertices into classes that do not share predecessors. We then apply Lemma~\ref{lem:pseudo:Chernoff:tuple} to those classes to finish the proof of the claim. Here come the details. 

\medskip

We partition $L_i^*(v)$ into predecessor disjoint sets. To do so we construct an auxiliary graph on vertex set $L_i^*(v)$ that has an edge $xx'$ for exactly those vertices $x\neq x'$ that share at least one predecessor in $H$. As~$H$ is $a$-arrangeable, the maximum degree of this auxiliary graph is bounded by $a\Delta(H)-1$. Hence we can apply Theorem~\ref{thm:HajnalSzemeredi} to partition the vertices of this auxiliary graph into stable sets $K_1\dcup\dots\dcup K_b$ with 
\begin{align} \label{eq:aux:rmd:Kl:1}
  |K_\ell| \ge \frac{|L_i^*(v)|}{a\Delta(H)} \ge \frac{2^{-a^2-1}\lambda \sqrt{n}}{a\,\kappa\, r}\log n\geByRef{eq:consts:n} 48\left(\frac{3}{\delta}\right)^a\log n_i
\end{align}
for $\ell \in [b]$. Those sets are predecessor disjoint in $H$. We now want to apply Lemma~\ref{lem:pseudo:Chernoff:tuple}. Let $\cI=\{ N^-(x):x\in K_\ell\}$. The sets in $\cI$ are  pairwise disjoint and have at most $a$ elements each. Name the elements of $\bigcup_{I\in \cI}I=\{y_1,\dots,y_s\}$ (with $s=|\bigcup_{I\in\cI}I|$) in ascending order with respect to the arrangeable ordering. Furthermore, let $\cA_k$ be a random variable which is 1 if and only if $y_k$ gets embedded into $N_G(v)$. By the definition of $L_i^*(v)$, the event $\cA_v(y_k)$ holds for each $k\in[s]$. It follows from the definition of $\cA_v(y_k)$ that
\begin{equation*} 
  \PP[\cA_k=1] = \PP[\phi(y_k) \in N_G(v)] = \frac{|A(y_k)\cap N_G(v)|}{|A(y_k)|} \ge \frac23\frac{|N_G(v)\cap V_j|}{|V_j|} \geByRef{eq:min-deg} \frac\delta3\;.
\end{equation*}
This lower bound on the probability of $\cA_k=1$ remains true even if we condition on other events $\cA_j=1$ (or their complements $\cA_j=0$), because in this calculation the lower bound relies solely on $|A(y_k)\cap N_G(v)|/|A(y_k)|$, which is at least $\delta/3$ for all $k\in[m]$ regardless of the embedding of other $y_j$. Hence, 
\[
  \PP\left[\cA_k=1 \Big| \parbox{150pt}{$\cA_{j}=1$ for all $j\in  J$\\ $\cA_j=0$ for all $j\in [k-1]\setminus J$}\right]\ge \frac\delta3
\]
for every $k$ and every $J\subseteq [k-1]$ (this is stronger than the condition required by Lemma~\ref{lem:pseudo:Chernoff:tuple}). By Lemma~\ref{lem:pseudo:Chernoff:tuple}, we have
\begin{align*}
& \PP\left[ \big|\big\{x\in K_\ell:\phi(N^-(x))\subseteq
        N_G(v)\big\}\big| 
   \ge \frac12\Big(\frac{\delta}{3}\Big)^a|K_\ell|\right]\\
& \hspace{3cm} = \PP\left[ \big|\big\{I\in \cI: \cA_i=1\text{ for all $i\in I$}\big\}\big| 
   \ge \frac12\Big(\frac{\delta}{3}\Big)^a|K_\ell|\right] \\
& \hspace{3cm} \ge 1-2\exp\left(-\frac1{12}\Big(\frac{\delta}{3}\Big)^a|K_\ell|\right) \\
& \hspace{3cm} \geByRef{eq:aux:rmd:Kl:1} 1-2\exp(-4\log n_i) = 1-2\cdot n_i^{-4}\,.
\end{align*}
Applying a union bound over all $\ell\in[b]$ we conclude that
\begin{align*}
  \deg_{F_i(T)}(v,L_i) &\ge \big|\big\{x \in L_i^*(v): \phi(N^-(x))\subseteq N_G(v)\big\}\big|\\
  &\ge \sum_{\ell\in[b]} \frac12\Big(\frac{\delta}{3}\Big)^a|K_\ell| =
  \frac12\Big(\frac{\delta}{3}\Big)^{a} |L_i^*(v)| \ge \frac{\delta^a}{2\cdot 2^{a^2+1}3^a}\lambda n_i \geByRef{eq:consts:eps:p} 3\sqrt{\eps'} n_i
\end{align*}
with probability at least $1-2n_i\exp(-4\log n_i)\ge 1-n_i^{-2}$.
\end{claimproof}
This concludes the proof of the lemma.
\end{proof}

The remainder of this section is dedicated to the proof of Lemma~\ref{lem:aux:rmd:Avy}. This proof will use similar ideas as the proof of Lemma~\ref{lem:aux:rmd:pre}. This time, however, we are not only interested in the predecessors of $x\in L_i^*$ but in the predecessors of the predecessors. We call those \emph{predecessors of second order} and say two vertices $x$, $x'$ are \emph{predecessor disjoint of second order} if $N^-(x)\cap N^-(x')=\emptyset$ and $N^-(N^-(x))\cap N^-(N^-(x'))=\emptyset$.

\medskip

To prove Lemma~\ref{lem:aux:rmd:Avy}, we have to show that for any vertex $v\in V_i$ many vertices $x$ in $L_i^*$ are such that all their predecessors $y\in N^-(x)$ are likely for $v$. Note that $x\in L_i^*$ implies that $y\in N^-(x)$ gets embedded into the special candidate set $C_{t(y),y}^\SP$. 

It depends only on the embedding of the vertices in $N^-(y)$ whether a given vertex $v\in V_i$ is likely for $y$ or not. Therefore, we formulate an event $\cB_{v,x}(z)$, which, if satisfied for all $z\in N^-(y)$, will imply $\cA_v(y)$ as we will show in the next proposition. Recall that $C_{1,y}^\SP=V_j^\SP$ for $y\in N^-(L_i^*)\subseteq X^*$ and $C_{t(z)+1,y}^\SP=C_{t(z),y}^\SP\cap N_G(\phi(z))$. For $x\in L_i^*$ and $z\in N^-(N^-(x))$ let $\cB_{v,x}(z)$ be the event that 
\[
\left|\frac{|C_{t(z),y}^\SP\cap N_G(v)|}{|C_{t(z),y}^\SP|} - \frac{|C_{t(z)+1,y}^\SP\cap N_G(v)|}{|C_{t(z)+1,y}^\SP|} \right|\le \frac{2\eps}{\delta-\eps}
\]
for all $y\in N^-(x)$. 

\begin{prop} \label{prop:hdp:predecessors} 
Let $i\in[r]$, $v\in V_i$, $x \in L_i^*$, and $z\in N^-(N^-(x))$, then 
\[ \PP\left[\cB_{v,x}(z) \big|\big. \cB_{v,x}(z')\text{ for all $z'\in N^-(N^-(x))$ with $t(z')<t(z)$}\right] \ge 1/2. \]
This remains true if we additionally condition on other events $\cB_{v,\tilde x}(\tilde z)$ (or their complements) with $\tilde z\in N^-(N^-(\tilde x))$ for $\tilde x\in L_i^*$, as long as $x$ and $\tilde x$ are predecessor disjoint of second order.\\
Moreover, if $\cB_{v,x}(z)$ occurs for all $z\in N^-(N^-(x))$, then $\cA_v(y)$ occurs for all $y\in N^-(x)$.
\end{prop}

\begin{proof}[of Proposition~\ref{prop:hdp:predecessors}]
Let $x\in L_i^*$ and let $z\in N^-(N^-(x))$ lie in $X_\ell$. Further assume that $\cB_{v,x}(z')$ holds for all $z'\in N^-(N^-(x))$ with $t(z')<t(z)$. For $y\in N^-(x)$ let $j(y)$ be such that $y\in X_{j(y)}$. Then $\cB_{v,x}(z')$ for all $z'\in N^-(y)$ with $t(z')<t(z)$ implies
\begin{align*}
\frac{|C_{t(z),y}^\SP\cap N_G(v)|}{|C_{t(z),y}^\SP|} &\ge \frac{|C_{1,y}^\SP\cap N_G(v)|}{|C_{1,y}^\SP|}-\frac{2\eps\cdot a}{\delta-\eps} \\
&= \frac{|V_{j(y)}\cap N_G(v)|}{|V_{j(y)}|}-\frac{2\eps\cdot a}{\delta-\eps} \geByRef{eq:min-deg} \frac\delta2 - \frac{2\eps\cdot a}{\delta-\eps}
\end{align*}
where the identity $C_{1,y}=V_{j(y)}$ is due to $y\notin S$. Hence $|C_{t(z),y}^\SP\cap N_G(v)|\ge \eps n_{j(\ell)}$ by~\eqref{eq:cond:size-special} and our choice of constants. Now fix a $y\in N^-(x)$. As $(V_{j(y)},V_\ell)$ is an $\eps$-regular pair all but at most $4\eps n_\ell$ vertices $w\in A_{t(z),z}\subseteq V_{\ell}$ simultaneously satisfy
\begin{align*}
\left|\frac{\big|N_G\big(w,C_{t(z),y}^\SP\cap N_G(v)\big)\big|}{|C_{t(z),y}^\SP\cap N_G(v)|} - d(V_{j(y)},V_\ell)\right| &\le \eps\,,\text{ and}\\
\left|\frac{|N_G(w, C_{t(z),y}^\SP)|}{|C_{t(z),y}^\SP|} - d(V_{j(y)},V_\ell)\right|&\le\eps\;.
\end{align*}
Hence, all but at most $4\eps a n_\ell$ vertices in $V_\ell$ satisfy the above inequalities for all $y\in N^-(x)$. If $\phi(z)=w$ for a vertex $w$ that satisfies the above inequalities for all $y\in N^-(x)$ we have
\begin{align*}
 \left|\frac{|C_{t(z)+1,y}^\SP\cap N_G(v)|}{|C_{t(z),y}^\SP\cap N_G(v)|} - \frac{|C_{t(z)+1,y}^\SP|}{|C_{t(z),y}^\SP|}\right| \le 2\eps\, .
\end{align*}
This implies $B_{v,x}(z)$ as
\begin{align*}
 \left|\frac{|C_{t(z)+1,y}^\SP\cap N_G(v)|}{|C_{t(z)+1,y}^\SP|} - \frac{|C_{t(z),y}^\SP\cap N_G(v)|}{|C_{t(z),y}^\SP|}\right| &\le 2\eps \frac{|C_{t(z),y}^\SP\cap N_G(v)|}{|C_{t(z)+1,y}^\SP|}\\ 
 &\le  2\eps \frac{|C_{t(z),y}^\SP|}{|C_{t(z)+1,y}^\SP|} \leByRef{eq:cond:size-ordinary} \frac{2\eps}{\delta-\eps}\, .
\end{align*}
Since $\phi(z)$ is chosen randomly from $A(z)\subseteq A_{t(z),z}$ with $|A(z)|\ge (\gamma/2) n_\ell$ by Lemma~\ref{lem:RGA:choices}, we obtain
\begin{align*}
\PP\left[ \cB_{v,x}(z) \big|\big. \cB_{v,x}(z')\text{ for all $z'\in N^-(N^-(x))$, $t(z')<t(z)$}\right] \ge 1- \frac{4\eps a n_\ell}{(\gamma/2) n_\ell} \ge \frac12 \,.
\end{align*}
Note that this probability follows alone from the $\eps$-regularity of the pairs $(V_{j(y)},V_\ell)$ and the fact that $A(z)$ and $C_{t(z),y}^\SP\cap N_G(v)$ are large. If $x$ and $\tilde x$ are predecessor disjoint of second order the outcome of the event $\cB_{v,\tilde x}(\tilde z)$ for $\tilde z\in N^-(N^-(\tilde x))$ does not influence those parameters. We can therefore condition on other events $\cB_{v,\tilde x}(\tilde z)$ as long as $x$ and $\tilde x$ are predecessor disjoint of second order.

\medskip

It remains to show the second part of the proposition, that $v$ is likely for all $y\in N^-(x)$ if $\cB_{v,x}(z)$ holds for all $z\in N^-(N^-(x))$. Again let $x\in L_i^*$ and let $y\in N^-(x)$ lie in $X_j$. Recall that condition~\eqref{eq:cond:size:special-neighbours} in the definition of the RGA guarantees
\[
    \left| \frac{|V_j^\SP\cap N_G(v)|}{|V_j^\SP|} - \frac{|V_j\cap N_G(v)|}{|V_j|}\right| \le \eps\,.
\]
 Moreover, $\cB_{v,x}(z)$ for all $z\in N^-(y)$, $C_{1,y}^\SP=V_j^\SP$ (as $y\notin S$) and the fact that $|N^-(y)|\le a$ imply
\begin{align*} \label{eq:hdp:intersection}
 \left|\frac{|C_{t(y),y}^\SP\cap N_G(v)|}{|C_{t(y),y}^\SP|} - \frac{|V_j^\SP\cap N_G(v)|}{|V_j^\SP|}\right| \le \frac{2\eps\cdot a}{\delta-\eps}\, .
\end{align*}
As $2\eps a/(\delta-\eps)+\eps \le \delta/36 \le (\delta/18)|V_j\cap N_G(v)|/|V_j|$ we conclude that 
\begin{align} \label{eq:rmd:CV}
\frac{|C_{t(y),y}^\SP \cap N_G(v)|}{|C_{t(y),y}^\SP|} \ge \frac{17}{18}\frac{|V_j\cap N_G(v)|}{|V_j|}
\end{align}
for all $y\in N^-(x)$ if $\cB_{v,x}(z)$ for all $z\in N^-(N^-(x))$. Equation~\eqref{eq:rmd:CV} in turn implies $\cA_v(y)$ as only few vertices get embedded into $V^\SP$ thus making $C_{t,y}^\SP \approx A_{t,y}^\SP$. More precisely, by Lemma~\ref{lem:RGA:choices} we have
\begin{align*}
 \frac{|A(y)\cap N_G(v)|}{|A(y)|} &\ge  \frac{|A_{t(y),y}^\SP\cap N_G(v)| - |A_{t(y),y}^\SP\setminus A(y)|}{|A_{t(y),y}^\SP|} \\
 &\ge \frac{|C_{t(y),y}^\SP\cap N_G(v)| - |X_j^*| - |Q_j(t(y))| - |A_{t(y),y}^\SP\setminus A(y)|}{|C_{t(y),y}^\SP|}\\
 &\ge \frac{|C_{t(y),y}^\SP\cap N_G(v)|}{|C_{t(y),y}^\SP|}-\frac{\delta}{18}\\
 &\geByRef{eq:rmd:CV} \left( \frac{17}{18} -\frac{2}{18}\right) \frac{|V_j\cap N_G(v)|}{|V_j|}\;.
\end{align*}
\end{proof}

We have seen that $x\in L_i^*$ also lies in $L_i^*(v)$ if $\cB_{v,x}(z)$ holds for all $z\in N^-(N^-(x))$. To prove Lemma~\ref{lem:aux:rmd:Avy} it therefore suffices to show that an arbitrary vertex $v$ has a linear number of vertices $x\in L_i^*$ with $\cB_{v,x}(z)$ for all $z\in N^-(N^-(x))$. 

\begin{proof}[of Lemma~\ref{lem:aux:rmd:Avy}]
Let $i\in[r]$ and $v\in V_i$ be arbitrary. We partition $L_i^*$ into classes of vertices that are predecessor disjoint of second order. Observe that for every $x\in L_i^*$ we have
\begin{align*}
  \left|\left\{ x'\in L_i^*: \parbox{160pt}{$N^-(x)\cap N^-(x')\neq \emptyset$ or \\[.5ex]$N^-(N^-(x))\cap N^-(N^-(x'))\neq \emptyset$} \right\}\right| &\le (a\Delta(H))^2 \le \frac{a^2n}{\log^2 n}\\
 & \lByRef{eq:consts:n} \frac{\lambda n_i}{36\cdot 2^{a^2}\log n_i}
\end{align*}
as $H$ is $a$-arrangeable. Recall that $|L_i^*|=\lambda n_i$. Therefore, Theorem~\ref{thm:HajnalSzemeredi} gives a partition $L_i^*=K_1\dcup\dots\dcup K_b$ with 
\begin{align} \label{eq:aux:rmd:Kl:2}
  |K_\ell|\ge 36\cdot 2^{a^2}\log n_i
\end{align}
for all $\ell\in[b]$ such that the vertices in $K_\ell$ are predecessor disjoint of second order. 

Next we want to apply Lemma~\ref{lem:pseudo:Chernoff:tuple}. Let $\ell\in[b]$ be fixed. We define $\cI=\{ N^-(N^-(x)): x\in K_\ell\}$. These sets are pairwise disjoint and have at most $a^2$ elements each. Name the elements of $\bigcup_{I\in \cI} I=\{z_1,\dots,z_{|\cup_{I\in \cI} I|}\}$ in ascending order with respect to the arrangeable ordering. Then for every $I\in \cI$ and every $z_k \in  I$ we have 
\[
  \PP\left[ \cB_{v,x}(z_k) \Big|\, \parbox{190pt}{ $\cB_{v,x}(z_j)$ for all $z_j\in J$,\\[.5ex] $\overline{\cB_{v,x}}(z_j)$ for all $z_j\in \{z_1,\dots,z_{k-1}\}\setminus J$}\right] \ge \frac 12
\] 
for every $J \subseteq \{z_1,\dots,z_{k-1}\}$ with $\{z_1,\dots,z_{k-1}\}\cap I \subseteq J$ by Proposition~\ref{prop:hdp:predecessors}. We set $K_\ell(v)\deff \left\{ x \in K_\ell: \cB_{v,x}(z) \text{ for all }z \in N^-(N^-(x))\right\}$ and apply Lemma~\ref{lem:pseudo:Chernoff:tuple} to derive 
\begin{align*}
   \PP\left[ \left|K_\ell(v)\right| \ge 2^{-a^2-1}|K_\ell|\right] &\ge 1 - 2\exp\left(-\frac{1}{12}2^{-a^2}|K_\ell|\right)\\
&\geByRef{eq:aux:rmd:Kl:2} 1 - 2\exp\left(-3\log n_i \right)=1-2\cdot n_i^{-3}\,.
\end{align*}
Note that we have $\bigcup_{\ell\in[b]} K_\ell(v) \subseteq L_i^*(v)$ as the following is true for every $x\in K_\ell$ by Proposition~\ref{prop:hdp:predecessors}: $\cB_{v,x}(z)$ for all $z\in N^-(N^-(x))$ implies $\cA_v(y)$ for all $y\in N^-(x)$. Taking a union bound over all $\ell \in [b]$ we thus obtain that
\[
   \PP\left[ \left|L_i^*(v)\right| \ge 2^{-a^2-1}|L_i^*|\right] \ge 1-b\cdot 2n_i^{-3} \ge 1 - \frac{2}{n_i^2}\,.
\] One further union bound over all $i\in[r]$ and $v\in V_i$ finishes the proof.
\end{proof}

\subsection{Proof of Theorem~\ref{thm:Blow-up:arr:full}} \label{sec:Proof:Blow-up:arr:full}

Putting everything together, we conclude that the RGA gives a spanning embedding of $H$ into $G$ with probability at least~$1/3$.
We now use Lemma~\ref{lem:reg:weighted:matching}, Proposition~\ref{prop:aux:left-min-deg}, and Lemma~\ref{lem:aux:right-min-deg} to prove our main result.

\begin{proof}[of Theorem~\ref{thm:Blow-up:arr:full}]
Let integers $C,a,\Delta_R,\kappa$ and $\delta,c>0$ be given. Set $a'= 5a^2\kappa\Delta_R$ and $\mu=1/(10a'(\kappa\Delta_R)^2)$. We invoke Theorem~\ref{thm:Blow-up:arr:almost-span} with parameters $C,a',\Delta_R,\kappa$ and $\delta,c,\mu>0$ to obtain $\eps,\alpha>0$. Let $r$ be given and choose $n_0$ as in Theorem~\ref{thm:Blow-up:arr:almost-span}.

Now let $R$ be a graph on $r$ vertices with $\Delta(R)<\Delta_R$. And let
$G$ and $H$ satisfy the conditions of Theorem~\ref{thm:Blow-up:arr:full},
i.e., let $G$ have the $(\eps,\delta)$-super-regular $R$-partition
$V(G)=V_1\dcup\dots\dcup V_r$ and let $H$ have a $\kappa$-balanced
$R$-partition $V(H)=X_1\dcup\dots\dcup X_r$. Further let
$\{x_1,\dots,x_n\}$ be an $a$-arrangeable ordering of~$H$. We apply
Lemma~\ref{lem:arr:StableEnding} to find an $a'$-arrangeable ordering
$\{x_1',\dots,x_n'\}$ of $H$ with a stable ending of order $\mu n$. Let $H'
= H[\{x_1',\dots,x_{(1-\mu)n}'\}]$ be the subgraph induced by the first
$(1-\mu)n$ vertices of the new ordering. We take this ordering and run the
RGA as described in Section~\ref{sec:rga:definition} to embed~$H'$
into~$G$. 
By Lemma~\ref{lem:rga:success} we have
\begin{align}\label{eq:aux:reg:1}
   \PP\big[\text{RGA successful and } \cR_i\text{ for all
       $i\in[r]$}\big] \ge \frac23 \,,
\end{align}
where~$\cR_i$ is the event that the auxiliary graph $F_i(t)$ is weighted
$\eps'$-regular for all $t\le T$.  Note that every image restricted vertex
$x\in S_i\cap V(H')$ has been embedded into $I(x)$ by the definition of the
RGA.

Now assume that $\cR_i$ holds for all $i\in[r]$.  It remains to embed
the stable set $L_i=V(H)\setminus V(H')$.
To this end we shall find in each $F_i^*:=F_i(T)[L_i\dcup V_i^{\Free}(T)]$ a
perfect matching, which defines a bijection from $L_i$ to
$V_i^{\Free}(T)$ that maps every $x\in L_i$ to a vertex $v\in
V_i^{\Free}(T)\cap C_{T,x}$. Note that again $x\in S_i$ is embedded into
$I(x)$ by construction. 

Since $F_i(T)$ is weighted $\eps'$-regular the subgraph $F_i^*$ is weighted $(\eps'/\mu)$-regular by Proposition~\ref{prop:regular:sub}. Moreover
\begin{align} \label{eq:aux:reg:2}
   \PP\left[\delta(F_i^*) \ge 3\sqrt{\eps'}n_i\text{ for all $i\in[r]$}\right] \ge \frac23
\end{align}
by Proposition~\ref{prop:aux:left-min-deg} and Lemma~\ref{lem:aux:right-min-deg}. In other words, with probability at least $2/3$ all graphs $F_i^* = F_i(T)[L_i \dcup V_i^\Free]$ are balanced, bipartite graphs on $2 \mu n_i$ vertices with $\deg(x)\ge 3\sqrt{\eps'/\mu} (\mu n_i)$ for all $x\in L_i\cup V_i^\Free$. Also note that $\omega(x)\ge \delta^a \ge \sqrt{\eps'/\mu}$ for all $x\in L_i$ by definition of $\eps'$. We conclude from Lemma~\ref{lem:reg:weighted:matching} that $F_i^* $ has a perfect matching if $F_i^*$ has minimum degree at least $3\sqrt{\eps'}n_i$. Hence, combining \eqref{eq:aux:reg:1} and~\eqref{eq:aux:reg:2} we obtain that the RGA terminates successfully and all $F_i^*$ have perfect matchings with probability at least $1/3$.
Thus there is an almost spanning embedding of $H'$ into $G$ that can be extended to a spanning embedding of $H$ into $G$.
\end{proof}

\subsection{Proof of Theorem~\ref{thm:Blow-up:arr:ext}} \label{sec:Proof:Blow-up:arr:ext}

We close this section by sketching the proof of Theorem~\ref{thm:Blow-up:arr:ext}, which is very similar to the proof of Theorem~\ref{thm:Blow-up:arr:full}. We start by quickly summarising the latter. 
For two graphs $G$ and $H$ let the partitions $V=V_1\dcup\dots\dcup V_k$ and $X=X_1\dcup\dots\dcup X_k$ satisfy the requirements of Theorem~\ref{thm:Blow-up:arr:full}. In order to find an embedding of $H$ into $G$ that maps the vertices of $X_i$ onto $V_i$ we proceeded in two steps. First we used a randomised greedy algorithm to embed an almost spanning part of $H$ into $G$. This left us with sets $L_i\subseteq X_i$ and $V_i^\Free\subseteq V_i$. We then found a bijection between the $L_i$ and $V_i^\Free$ that completed the embedding of $H$. 

\smallskip

More precisely, we did the following. We ran the randomised greedy algorithm from Section~\ref{sec:rga:definition} and defined auxiliary graphs $F_i(t)$ on vertex sets $V_i\dcup X_i$ that kept track of all possible embeddings at time $t$ of the embedding algorithm. We showed that the randomised greedy embedding succeeds for the almost spanning subgraph if all the auxiliary graphs remain weighted regular (Lemma~\ref{lem:reg-small-queue}). This in turn happens with probability at least $2/3$ by Lemma~\ref{lem:rga:success}. This finished stage one of the embedding (and also proved Theorem~\ref{thm:Blow-up:arr:almost-span}).

For the second stage of the embedding we assumed that stage one found an almost spanning embedding by time $T$ and that all auxiliary graphs are weighted regular. We defined $F_i^*(T)$ to be the subgraph of $F_i(T)$ induced by $L_i\dcup V_i^\Free$. This subgraph inherits (some) weighted regularity from $F_i(T)$. Moreover, we showed that all $F_i^*(T)$ have a minimum degree which is linear in $n_i$ with probability at least $2/3$ (see Proposition~\ref{prop:aux:left-min-deg} and Lemma~\ref{lem:aux:right-min-deg}). Each $F_i^*(T)$ has a perfect matching in this case by Lemma~\ref{lem:reg:weighted:matching}. Those perfect matchings then gave the bijection of $L_i$ onto $V_i^\Free$ that completed the embedding of $H$ into $G$. We concluded that with probability at least $2/3$ the almost spanning embedding found by the randomised greedy algorithm in stage one can be extended to a spanning embedding of $H$ into $G$.

\medskip

For the proof of Theorem~\ref{thm:Blow-up:arr:ext} we proceed in exactly the same way. Note that Theorem~\ref{thm:Blow-up:arr:full} and Theorem~\ref{thm:Blow-up:arr:ext} differ only in the following aspects. The first allows a maximum degree of $\sqrt{n}/\log n$ for $H$ while the latter extends this to $\Delta(H)\le \xi n/\log n$. This does not come free of charge. Theorem~\ref{thm:Blow-up:arr:ext} not only requires the $R$-partition of $G$ to be super-regular but also imposes what we call the \emph{tuple condition}, that every tuple of $a+1$ vertices in $V\setminus V_i$ have a linearly sized joint neighbourhood in $V_i$. We now sketch how one has to change the proof of Theorem~\ref{thm:Blow-up:arr:full} to obtain Theorem~\ref{thm:Blow-up:arr:ext}.

\smallskip

Again we proceed in two stages. The first of those, which gives the almost spanning embedding, is identical to the previously described one: here the larger maximum degree is not an obstacle (see also Theorem~\ref{thm:Blow-up:arr:almost-span}). 
Again all auxiliary graphs are weighted regular by the end of the \textsc{Embedding Phase} with probability at least $2/3$. Moreover, all vertices in $L_i$ have linear degree in $F_i^*(T)$ by the same argument as before (see Proposition~\ref{prop:aux:left-min-deg} and its proof).

It now remains to show to show that every vertex $v$ in $V_i^\Free$ has a linear degree in the auxiliary graph $F_i^*(T)$. At this point we deviate from the proof of Theorem~\ref{thm:Blow-up:arr:full}. Recall that $L_i^*$ was defined as the last $\lambda n_i$ vertices of $X_i\setminus N(S)$ in the arrangeable ordering and $L_i^*(v)$ was defined as the set of vertices $x\in L_i^*$ with $\cA_v(y)$ for all $y\in N^-(x)$. We still want to prove that $L_i^*(v)$ is large for every $v$ as this again would imply the linear minimum degree for all $v\in V_i^\Free$. However, the maximum degree $\Delta(H)\le \xi n/\log n$ does not allow us to partition $L_i$ into sets which are predecessor disjoint of second order any more. This, however, was crucial for our proof of $|L_i^*(v)|\ge 2^{-a^2-1}|L_i^*|$ (see the proof of Lemma~\ref{lem:aux:rmd:Avy}). 

\smallskip

We may, however, alter the definition of the event $\cA_v(y)$ to overcome this obstacle. Instead of requiring that $|A(y) \cap N_G(v)|/|A(y)|\ge (2/3)|V_j\cap N_G(v)|/|V_j|$, we now define $\cA_v(y)$ in the proof of Theorem~\ref{thm:Blow-up:arr:ext} to be the event that
\begin{align*} 
   \frac{|A(y)\cap N(v)|}{|A(y)|} \ge \frac\iota2\;.
\end{align*}
We still denote by $L_i^*(v)$ the set of vertices $x\in L_i^*$ with $\cA_v(y)$ for all $y\in N^-(x)$. Now the tuple condition guarantees that $|C_{t(y),y}\cap N_G(v)| \ge \iota n_j$ for any $y\in X_j$ and $v \in V\setminus V_j$. Since we chose $V_j^\SP\subseteq V_j$ randomly we obtain $|C_{t(y),y}^\SP\cap N_G(v)| \ge (\mu/20) \iota n_j$ for all $y\in X_j$ \emph{almost surely}. The same arguments as in the proof of Proposition~\ref{prop:hdp:predecessors} imply that $A(y)\approx C_{t(y),y}^\SP$ for all $y$ that are predecessors of vertices $x\in L_i^*$. Hence, 
\[
 \frac{|A(y)\cap N_G(v)|}{|A(y)|} \approx \frac{|C_{t(y),y}^\SP\cap N_G(v)|}{|C_{t(y),y}^\SP|}\approx \frac{(\mu/20) \iota n_j}{(\mu/10) \delta^a n_j} \ge \frac\iota2
\]
for all $x\in L_i^*$ and all $y\in N^-(x)$ almost surely. If this is the case we have $|L_i^*(v)|=|L_i^*|=\lambda n_i$ and therefore the assertion of Lemma~\ref{lem:aux:rmd:Avy} holds also in this setting. It remains to show that the same is true for Lemma~\ref{lem:aux:rmd:pre}. Indeed, after some appropriate adjustments of the constants, the very same argument implies $\deg_{F_i(T)}(v,L_i)\ge 3\eps'n_i$ for all $i\in[r]$ and $v\in V_i^\Free$ if $|L_i^*(v)|=|L_i^*|$. More precisely, the change in the definition of $\cA_v(y)$ will force smaller values of $\eps'$, that is, the constant in the bound of the joint neighbourhood of each $(a+1)$-tuple has to be large compared to the $\eps$ in the $\eps$-regularity of the partition $V_1\dcup\dots\dcup V_k$. The constants then relate as 
\[
  0 < \xi \ll \eps \ll \eps' \ll \lambda \ll \gamma \ll \mu, \delta, \iota \le 1\, .
\]
The remaining steps in the proof of Theorem~\ref{thm:Blow-up:arr:ext} are identical to those in the proof of Theorem~\ref{thm:Blow-up:arr:full}. For $i\in[r]$ the minimum degree in $F_i^*(T)$ together with the weighted regularity implies that $F_i^*(T)$ has a perfect matching. The perfect matching defines a bijection of $L_i$ onto $V_i^\Free$ that in turn completes the embedding of $H$ into $G$.

\medskip

To wrap up, let us quickly comment on the different degree bounds for $H$ in Theorem~\ref{thm:Blow-up:arr:full} and Theorem~\ref{thm:Blow-up:arr:ext}. The proof of Theorem~\ref{thm:Blow-up:arr:ext} just sketched only requires $\Delta(H)= \xi n/\log n$. This is needed to partition $L_i^*$ into \emph{predecessor disjoint sets} in the last step in order to prove the minimum degree for the auxiliary graphs.

Contrary to that the proof of Theorem~\ref{thm:Blow-up:arr:full} partitions the vertices of $L_i^*$ into sets which are \emph{predecessor disjoint of second order}, i.e., which do have $N^-(N^-(x))\cap N^-(N^-(x'))=\emptyset$ for all $x\neq x'$. This is necessary to ensure that there is a linear number of vertices $x$ in $L_i^*$ with $\cA_v(y)$ for all $y\in N^-(x)$, i.e., whose predecessors all get embedded into $N_G(v)$ with probability $\delta/3$. More precisely we ensure that, all predecessors $y$ of $x$ have the following property. The predecessors $z_1,\dots,z_k$ of $y$ are embedded to a $k$-tuple $\left(\phi(z_1),\dots,\phi(z_k)\right)$ of vertices in $G$ such that $\bigcap N(\phi(z_i)) \cap N_G(v) \cap V_{j(y)}$ is large. This fact follows trivially from the tuple condition of Theorem~\ref{thm:Blow-up:arr:ext} and hence we don't need a partition into predecessor disjoint sets of second order.

\section{Optimality}
\label{sec:opt}

The aim of this section is twofold. Firstly, we shall investigate why the
degree bounds given  in Theorem~\ref{thm:Blow-up:arr:full}  and in
Theorem~\ref{thm:Blow-up:arr:ext} are best possible. Secondly, we shall why
the conditions Theorem~\ref{thm:Blow-up:arr:full} imposes on  image restrictions
are so restrictive.

\paragraph{Optimality of Theorem~\ref{thm:Blow-up:arr:ext}.}

To argue that the requirement $\Delta(H)\le n/\log n$ is optimal up to
the constant factor we use a construction from~\cite{KSS95} and the following proposition.

\begin{prop} \label{prop:Domination}
For every $\eps>0$ the domination number of a graph $\Gnp$ with high probability is larger than $(1-\eps)p \log n$.
\end{prop}

\begin{proof}
The probability that a graph in $\Gnp$ has a dominating set of size $r$ is bounded by
\begin{align*}
 \binom nr\left(1-(1-p)^r\right)^{n-r} &\le \exp\left(r\log n - \exp(-rp)(n-r)\right)\,.\\
\intertext{Setting $r=(1-\eps) p \log n$ we obtain}
 \PP\left[\parbox{130pt}{$\Gnp$ has a dominating set of size $(1-\eps) p\log n$}\right] &\le \exp\left((1-\eps)p \log^2 n - \frac{n-(1-\eps)p\log n}{n^{1-\eps}}\right) \to 0
\end{align*}
for every (fixed) positive $\eps$.
\end{proof}

Let $H$ be a tree with a root of degree $\tfrac12\log n$, such that each neighbour of this root has $2n/\log n$ leaves as neighbours. This graph $H$ almost surely is not a subgraph of $\Gnp[0.9]$ by Proposition~\ref{prop:Domination} as the neighbours of the root form a dominating set.


\paragraph{Optimality of Theorem~\ref{thm:Blow-up:arr:full}.}

The degree bound $\Delta(H)\le\sqrt{n}/\log n$ is optimal up to the $\log$-factor. 
More precisely, we can show the following.

\begin{prop} \label{prop:UpperBound:deg} 
  For every $\eps>0$ and $n_0\in\N$ there are $n\ge
  n_0$, 
  an $(\eps,1/2)$-super-regular pair $(V_1,V_2)$ with $|V_1|=|V_2|=n$ and a
  tree $T\subseteq K_{n,n}$ with $\Delta(T)\le \sqrt n +1$, such that
  $(V_1,V_2)$ does not contain~$T$.
\end{prop}

Condition~\eqref{item:Blow-up:ir} of Theorem~\ref{thm:Blow-up:arr:full}
allows only a constant number of permissible image restrictions per cluster.
The following proposition shows that also this is best possible (up to the
value of the constant).

\begin{prop} \label{prop:UpperBound:landing} 
  For every $\eps>0$, $n_0\in\N$, and every $w\colon\N \to \N$ which goes
  to infinity arbitrarily slowly, there are $n\ge n_0$, an
  $(\eps,1/2)$-super-regular pair $(V_1,V_2)$ with $|V_1|=|V_2|=n$ and a
  tree $T\subseteq K_{n,n}$ with $\Delta(T)\le w(n)$ such that the following
  is true.  The images of $w(n)$ vertices of~$T$ can be restricted to sets
  of size $n/2$ in $V_1\cup V_2$ such that no embedding of~$T$ into $(V_1,V_2)$
  respects these image restrictions.
\end{prop}

We remark that our construction for
Proposition~\ref{prop:UpperBound:landing} does not require a spanning
tree~$T$, but only one on $w(n)+1$ vertices. Moreover, this proposition shows that
the number of admissible image restrictions drops from
linear (in the original Blow-up Lemma) to constant (in
Theorem~\ref{thm:Blow-up:arr:full}), if the maximum degree of the target
graph~$H$ increases from constant to an increasing function.

\smallskip

We now give the constructions that prove these two propositions.

\begin{proof}[of Proposition~\ref{prop:UpperBound:deg} (sketch)]
  Let $\eps>0$ and $n_0$ be given, choose an integer~$k$ such that $1/k \ll
  \eps$ and an integer~$n$ such that $k,n_0\ll n$, and consider the
  following bipartite graph $G_k=(V_1\dcup V_2,E)$ with $|V_1|=|V_2|=n$.
  Let $W_1,\dots,W_k$ be a balanced partition of $V_1$. Now for each
  \emph{odd} $i\in[k]$ we randomly and independently choose a subset
  $U_i\subseteq V_2$ of size $n/2$; and we set $U_{i+1}:= V_2\setminus
  U_i$. Then we insert exactly all those edges into~$E$ which have one
  vertex in~$W_i$ and the other in~$U_i$, for $i\in[k]$. Clearly, every
  vertex in~$G$ has degree $n/2$. In addition, using the degree co-degree
  characterisation of $\eps$-regularity it is not difficult to check that
  $(V_1,V_2)$ almost surely is $\eps$-regular.

  Next, we construct the tree~$T$ as follows. We start with a
  tree~$T'$, which consists of a root of degree $\sqrt n -1$ and is such
  that each child of this root has exactly $\sqrt{n}$ leaves as children.
  For obtaining~$T$, we then take two copies of~$T'$, call their
  roots~$x_1$ and~$x_2$, respectively, and add an edge between~$x_1$
  and~$x_2$. Clearly, the two colour classes of~$T$ have size~$n$ and
  $\Delta(T)=\sqrt n +1$.

  It remains to show that $T\not\subseteq G_k$. Assume for contradiction
  that there is an embedding~$\phi$ of~$T$ into $G_k$ such that
  $\phi(x_1)\in W_1$.  Note that $n-1$ vertices in~$T$ have distance~$2$
  from~$x_1$. Since~$G_k$ is bipartite~$\phi$ has to map these $n-1$ vertices
  to~$V_1$.  In particular, one of them has to be embedded
  in~$W_2$. However, the distance between~$W_1$ and~$W_2$ in~$G_k$ is
  greater than~$2$.
\end{proof}

The proof of Proposition~\ref{prop:UpperBound:landing} proceeds similarly.

\begin{proof}[of Proposition~\ref{prop:UpperBound:landing} (sketch)]
  Let~$\eps$, $n_0$ and $w$ be given, choose $n$ large enough so that
  $n_0\le n$ and $1/w(n) \ll \eps$, and set $k:=w(n)$.

  We reuse the graph~$G_k=(V_1\dcup V_2,E)$ from the previous proof as
  $\eps$-regular pair. Now consider any balanced tree~$T$ with a vertex~$x$
  of degree $\Delta(T)=w(n)=k$. Let $\{y_1,\dots,y_k\}$ be the neighbours
  of~$x$ in~$T$. For $i\in[k]$ we then restrict the image of~$y_i$ to
  $V_2\setminus U_i$.

  We claim that there is no embedding of~$T$ into~$G$ that respects these
  image restrictions. Indeed, clearly~$x$ has to be embedded
  into~$W_j\subseteq V_1$ for some $j\in[k]$ because its neighbours are image
  restricted to subsets of~$V_2$.  However, by the definition of~$U_j$
  this prevents~$y_j$ from being embedded into $V_2\setminus U_j$.
\end{proof}
\clearpage

\section{Applications} \label{sec:Applications}

\subsection{$F$-factors for growing degrees}

This section is to prove Theorem~\ref{thm:GrowingHfactors}. Our strategy will be to repeatedly embed a collection of copies of $F$ into a super-regular $r$-tuple in $G$ with the help of the Blow-up Lemma version stated as Theorem~\ref{thm:Blow-up:arr:full}. The following result by B\"ottcher, Schacht, and Taraz~\cite[Lemma 6]{BoeSchTar09} says that for $\gamma>0$ any sufficiently large graph $G$ with $\delta(G)\ge ((r-1)/r+\gamma)|G|$ has a regular partition with a reduced graph $R$ that contains a $K_r$-factor. Moreover, all pairs of vertices in $R$ that lie in the same $K_r$ span super-regular pairs in $G$.
Let $K_r^{(k)}$ denote the disjoint union of $k$ complete graphs on $r$ vertices each. For all $n,k,r\in\N$, we call an integer partition $(n_{i,j})_{i\in[k],j\in[r]}$ of $[n]$ (with $n_{i,j}\in\N$ for all $i\in[k]$ and $j\in[r]$) $r$-equitable, if $|n_{i,j}-n_{i,j'}|\le 1$ for all $i\in[k]$ and $j,j'\in[r]$.

\begin{lem}\label{lem:EquiSuperPartition}
For all $r\in\N$ and $\gamma>0$ there exists $\delta>0$ and $\eps_0>0$ such that for every positive $\eps\le\eps_0$ there exists $K_0$ and $\xi_0>0$ such that for all $n\ge K_0$ and for every graph $G$ on vertex set $[n]$ with $\delta(G)\ge((r-1)/r+\gamma)n$ there exists $k\in \N\setminus\{0\}$, and a graph $K_r^{(k)}$ on vertex set $[k]\times [r]$ with
\renewcommand{\labelenumi}{(R\arabic{enumi})}
\begin{enumerate}
\item $k\le K_0$,
\item there is an $r$-equitable integer partition $(m_{i,j})_{i\in[k],j\in[r]}$ of $[n]$ with $(1+\eps)n/(kr)\ge m_{i,j}\ge (1-\eps)n/(kr)$ such that the following holds.\footnote{The upper bound on $m_{i,j}$ is implicit in the proof of the lemma but not explicitly stated in~\cite{BoeSchTar09}.}
\end{enumerate}
For every partition $(n_{i,j})_{i\in[k],j\in[r]}$ of $n$ with $m_{i,j}-\xi_0 n \le n_{i,j} \le m_{i,j}+\xi_0 n$ there exists a partition $(V_{i,j})_{i\in[k],j\in[r]}$ of $V$ with
\renewcommand{\theenumi}{V\arabic{enumi}}
\renewcommand{\labelenumi}{(\theenumi)}
\begin{enumerate}
\item\label{eq:lem:ESP:1} $|V_{i,j}|=n_{i,j}$,
\item\label{eq:lem:ESP:2} $(V_{i,j})_{i\in[k],j\in[r]}$ is $(\eps,\delta)$-super-regular on $K_r^{(k)}$.
\end{enumerate}
\end{lem}
\renewcommand{\labelenumi}{(\roman{enumi})}

Using this partitioning result for $G$, Theorem~\ref{thm:GrowingHfactors} follows easily.

\begin{proof}[of Theorem~\ref{thm:GrowingHfactors}]
We alternatingly choose constants as given by Theorem~\ref{thm:Blow-up:arr:full} and Lemma~\ref{lem:EquiSuperPartition}. So let $\delta,\eps_0>0$ be the constants given by Lemma~\ref{lem:EquiSuperPartition} for $r$ and $\gamma>0$. Further let $\eps,\alpha>0$ be the constants given by Theorem~\ref{thm:Blow-up:arr:full} for $C=0$, $a$, $\Delta_R=r$, $\kappa=2$, $c=1$ and $\delta$. We are setting $C=0$ as we do not use any image restrictions in this proof. If necessary we decrease $\eps$ such that $\eps\le\eps_0$ holds. Let $K_0$ and $\xi_0>0$ be as in Lemma~\ref{lem:EquiSuperPartition} with $\eps$ as set before. For $r$ let $n_0$ be given by Theorem~\ref{thm:Blow-up:arr:full}. If necessary increase $n_0$ such that $n_0\ge K_0$. Finally set $\xi=\xi_0$. In the following we assume that 
\begin{enumerate}
\item $G$ is of order $n\ge n_0$ and has $\delta(G)\ge(\tfrac{r-1}{r}+\gamma)n$, and
\item $H$ is an $a$-arrangeable, $r$-chromatic $F$-factor with $|F|\le \xi n$, $\Delta(F)\le\sqrt{n}/\log n$.
\end{enumerate}
We need to show that $H\subseteq G$. For this purpose we partition $H$ into subgraphs $H_1,\dots,H_k$, where $H_i$ is to be embedded into $\cup_{j\in[r]}V_{i,j}$ later, as follows. Let $(m_{i,j})_{i\in[k],j\in[r]}$ be an $r$-equitable partition of $[n]$ with $m_{i,j}\ge (1-\eps)n/(kr)$ as given by Lemma~\ref{lem:EquiSuperPartition}. For $i=1,\dots,k-1$ we choose $\ell_i$ such that both
\begin{align}\label{eq:GrowingHfactors:partition:1}
 \big|m_{i,j}-\ell_i |F|\big| \le |F| \text{, \qquad and \qquad} \left|\sum_{i'\le i} (m_{i',j}-\ell_{i'}|F|)\right| \le |F|
\end{align}
for all $j\in[r]$. Let $n_{i,j}= \ell_i |F|$ for all $j\in[r]$ and set $H_i$ to be $\ell_ir$ copies of $F$. Note that there exists an $r$-colouring of $H_i$ in which each colour class $X_{i,j}$ has exactly $n_{i,j}$ vertices. Finally $H_k$ is set to be $H\setminus (H_1\cup\dots\cup H_{k-1})$. Let $\chi:V(H_k)\to [r]$ be a colouring of $H_k$ where the colour-classes have as equal sizes as possible and set $n_{k,j}\deff|\chi^{-1}(j)|$ and $X_{k,j}\deff \chi^{-1}(j)$ for $j\in[r]$. It follows from~\eqref{eq:GrowingHfactors:partition:1} that 
\begin{align} \label{eq:GrowingHfactors:partition:2}
  |n_{i,j}-m_{i,j}| \le |F| \le \xi_0 n
\end{align}
for all $i\in[k]$, $j\in[r]$. Thus there exists a partition $(V_{i,j})_{i\in[k],j\in[r]}$ of $V(G)$ with properties~\eqref{eq:lem:ESP:1} and~\eqref{eq:lem:ESP:2} by Lemma~\ref{lem:EquiSuperPartition}.

We apply Theorem~\ref{thm:Blow-up:arr:full} to embed $H_i$ into $G[V_{i,1}\cup\dots\cup V_{i,r}]$ for every $i\in[k]$. 
Note that we have partitioned $V(H_i)=X_{i,1}\cup\dots\cup X_{i,r}$ in such a way that $|X_{i,j}|=n_{i,j}$ and $vw\in E(H_i)$ implies $v\in X_{i,j}$, $w\in X_{i,j'}$ with $j\neq j'$. Now properties~\eqref{eq:lem:ESP:1},~\eqref{eq:lem:ESP:2} guarantee that $|V_{i,j}|=n_{i,j}$ and $(V_{i,j},V_{i,j'})$ is an $(\eps,\delta)$-super-regular pair in $G$ for all $i\in[k]$ and $j,j'\in[r]$ with $j\neq j'$. It follows that $H_i$ is a subgraph of $G[V_{i,1}\cup\dots\cup V_{i,r}]$ by Theorem~\ref{thm:Blow-up:arr:full}.
\end{proof}

\subsection{Random graphs and universality}

Next we prove Theorem~\ref{thm:GnpUniversal}, which states that $G=\Gnp$ is universal for the class of $a$-arrangeable bounded degree graphs, $\mathcal{H}_{n,a,\xi}= \{H: |H|=n, H \text{ $a$-arr., $\Delta(H)\le \xi n/\log n$}\}$.

To prove this we will find a balanced partition of $G$ and apply Theorem~\ref{thm:Blow-up:arr:ext}. For this purpose we also have to find a balanced partition of the graphs $H\in\mathcal{H}_{n,a,\xi}$. To this end we shall use the following result of Kostochka, Nakprasit, and Pemmaraju~\cite{KosNakPem05}.

\begin{thm}[Theorem 4 from~\cite{KosNakPem05}] \label{thm:DegenerateEquit}
Every $a$-arrangeable\footnote{In fact~\cite{KosNakPem05} shows this result for the more general class of $a$-degenerate graphs.} graph $H$ with $\Delta(H)\le n/15$ has a balanced $k$-colouring for each $k\ge 16a$.
\end{thm}

A graph has a \emph{balanced $k$-colouring} if the graph has a proper colouring with at most $k$ colours such that the sizes of the colour classes differ by at most 1.

\begin{proof}[of Theorem~\ref{thm:GnpUniversal}]
Let $a$ and $p$ be given. Set $\Delta_R\deff 16a$, $\kappa=1$, $\iota\deff \tfrac12 p^{a+1}$, $\delta\deff p/2$ and let $R$ be a complete graph on $16a$ vertices. Set $r\deff 16a$ and let $\eps$, $\xi$, $n_0$ as given by Theorem~\ref{thm:Blow-up:arr:ext}. Let $n\ge n_0$ and let $V=V_1\dcup\dots\dcup V_{r}$ be a balanced partition of $[n]$. Then we generate a random graph $G=\Gnp$ on vertex set $[n]$. Every pair $(V_i,V_j)$ is $(\eps, p/2)$-super-regular in $G$ with high probability. Furthermore with high probability we have that every tuple $(u_1,\dots,u_{a+1})\subseteq V\setminus V_i$ satisfies $|\cap_{j\in[a+1]}N_G(u_j)\cap V_i|\ge \iota |V_i|$. So assume this is the case and let $H\in\mathcal{H}_{n,a,\xi}$. We partition $H$ into $16a$ equally sized stable sets with the help of Theorem~\ref{thm:DegenerateEquit}. Thus $H$ satisfies the requirements of Theorem~\ref{thm:Blow-up:arr:ext} and $H$ embeds into $G$.
\end{proof}

\clearpage

\bibliographystyle{siam}   
\bibliography{ArrangeableBlow-up}  

\clearpage
\addcontentsline{toc}{section}{Appendix}
\section*{Appendix} 


\subsection*{Weighted regularity} 

In this section we provide some background on \emph{weighted regularity}. In particular, we supplement the proofs of Lemma~\ref{lem:reg:weighted:deg-codeg} and Lemma~\ref{lem:reg:weighted:matching}. We start with a short introduction to the results on weighted regularity by Czygrinow and R\"odl~\cite{CzygRodl00}. Their focus lies on hypergraphs. However, we only present the graph case here. 

Czygrinow and R\"odl define their weight function on the set of edges (whereas in our scenario we have a bipartite graph with weights on the vertices of one class). They consider weighted graphs $G=(V,\tilde\omega)$ where $\tilde\omega: V\times V\to \mathbb{N}_{\ge 0}$. One can think of $\tilde\omega(x,y)$ as the multiplicity of the edge $(x,y)$. Their weighted degree and co-degree for $x,y\in V$ are then defined as 
\[
 \deg_{\tilde\omega}^*(x) \deff \sum_{y\in V}\tilde\omega(x,y), \qquad \deg_{\tilde\omega}^*(x,y) \deff \sum_{z\in V}\tilde\omega(x,z)\tilde\omega(y,z)\,.
\]
For disjoint $A,B\subseteq V$ they define
\[ 
 d_{\tilde\omega}^*(A,B) = \frac{\sum\tilde\omega(x,y)}{K|A|\,|B|}\, ,
\]
where the sum is over all pairs $(x,y) \in A\times B$ and $K\deff 1+\max\{\tilde\omega(x,y): (x,y)\in V\times V\}$.\footnote{Czygrinow and R\"odl require $K$ to be strictly larger than the maximal weight for technical reasons.} A pair $(A,B)$ in $G=(V,\tilde\omega)$ with $A\cap B=\emptyset$ is called \emph{$(\eps,\tilde\omega)$-regular} if
\[
  |d_{\tilde\omega}^*(A,B)-d_{\tilde\omega}^*(A',B')|<\eps
\]
for all $A'\subseteq A$ with $|A'|\ge \eps |A|$ and all $B'\subseteq B$ with $|B'|\ge\eps|B|$. As in the unweighted case, regular pairs can be characterised by the degree and co-degree distribution of their vertices. 
The following lemma (see~\cite[Lemma 4.2]{CzygRodl00}) shows that a pair is weighted regular in the setting of Czygrinow and R\"odl if most of the vertices have the correct weighted degree and most of the pairs have the correct weighted co-degree.

\begin{lem}[Czygrinow, R\"odl~\cite{CzygRodl00}] \label{lem:weighted:CR}
Let $G=(A\dcup B,\tilde\omega)$ be a weighted graph with $|A|=|B|=n$ and let $\eps,\xi\in (0,1)$, $\xi^2<\eps$, $n\ge 1/\xi$. Assume that both of the following conditions are satisfied:
\begin{enumerate}
\renewcommand{\theenumi}{\roman{enumi}'}
\renewcommand{\labelenumi}{(\theenumi)}
\item\label{eq:reg:w:CR:1} $\left|\left\{x\in A:|\deg_{\tilde\omega}^*(x)-K\,d_{\tilde\omega}^*(A,B)n| > K\xi^2n\right\}\right| <\xi^2n$, and
\item\label{eq:reg:w:CR:2} $\left|\left\{ \{x_i,x_j\} \in \binom A2 :\left|\deg_{\tilde\omega}^*(x_i,x_j)-K^2d_{\tilde\omega}^*(A,B)^2n\right| \ge K^2\xi n\right\}\right|\le \xi\binom n2$.
\end{enumerate}
Then for every $A'\subseteq A$ with $|A'|\ge \eps n$ and every $B'\subseteq B$ with $|B'|\ge \eps n$ we have
\[
 |d_{\tilde\omega}^*(A',B') - d_{\tilde\omega}^*(A,B)|\le 2\frac{\xi^2}\eps + \frac{\sqrt{5\xi}}{\eps^2-\eps\xi^2}\,.
\]
\end{lem}

The assertion of Lemma~\ref{lem:weighted:CR} implies that the pair $(A,B)$ is $(\eps',\tilde\omega)$-regular, where $\eps'=\max\{\eps, 2\xi^2/\eps + \sqrt{5\xi}/(\eps^2-\eps\xi^2)\}$, if the conditions of the lemma are satisfied.

\medskip

Our goal is to translate this result into our setting of weighted regularity (see Section~\ref{sec:WeightedRegularity}). We shortly recall our definition of weighted graphs and weighted regularity before we restate and prove Lemma~\ref{lem:reg:weighted:deg-codeg}. 

Let $G=(A\dcup B,E)$ be a bipartite graph and $\omega: A\to [0,1]$ be our weight function for $G$. We define the weighted degree of a vertex $x\in A$ to be $\deg_\omega(x)= \omega(x)|N(x,B)|$ and the weighted co-degree of $x,y\in A$ as $\deg_\omega(x,y)=\omega(x)\omega(y)|N(x,B)\cap N(y,B)|$. Similarly, the weighted density of a pair $(A',B')$ is defined as 
\[
d_\omega(A',B')\deff\sum_{x\in A'}\frac{\omega(x)|N(x,B')|}{|A'|\cdot |B'|}\,.
\] 
Again the pair $(A,B)$ is called weighted $\eps$-regular if
\[
    |d_\omega(A,B) - d_\omega(A',B')|\le \eps
\]
for all $A'\subseteq A$ and $B'\subseteq B$ with $|A'|\ge \eps |A|$ and $|B'|\ge \eps |B|$. We now prove Lemma~\ref{lem:reg:weighted:deg-codeg}, which we restate here for the reader's convenience.

\begin{lemNN}[Lemma~\ref{lem:reg:weighted:deg-codeg}]
Let $\eps>0$ and $n\ge \eps^{-6}$. Further let $G=(A\dcup B,E)$ be a bipartite graph with $|A|=|B|=n$ and let $\omega: A \to [\eps,1]$ be a weight function for $G$. If
\begin{enumerate}
\renewcommand{\theenumi}{\roman{enumi}}
\renewcommand{\labelenumi}{(\theenumi)}
\item\label{eq:reg:w:d-g:1} $|\{x\in A: |\deg_\omega(x)-d_\omega(A,B)n| > \eps^{14}n\}| < \eps^{12}n$ \quad and
\item\label{eq:reg:w:d-g:2} $|\{\{x,y\}\in \binom A2: |\deg_\omega(x,y)-d_\omega(A,B)^2n| \ge \eps^{9}n\}| \le \eps^{6}\binom n2$
\end{enumerate}
then $(A,B)$ is a weighted $3\eps$-regular pair.
\end{lemNN}

\newcommand{\CST}{C}

\begin{proof}[of Lemma~\ref{lem:reg:weighted:deg-codeg}]
Let $\eps>0$, $G=(A\cup B,E)$ and $\omega:A\to[\eps,1]$ satisfy the requirements of the lemma. From this $\omega$ we define a weight function $\tilde\omega:A\times B \to \mathbb{N}_{\ge 0}$ in the setting of Lemma~\ref{lem:weighted:CR}. For $(x,y)\in A\times B$ we set 
\[
  \tilde\omega(x,y) \deff 
\begin{cases}
\left\lceil \CST \cdot \omega(x)\right\rceil & \text{ if $\{x,y\}\in E$,}\\
\, 0 & \text{ otherwise,} 
\end{cases}
\]
where $\eps^{-13}-1\le\CST\le\eps^{-14}$ is chosen such that $K=\max\{\tilde\omega(x,y)+1: (x,y)\in A\times B\}\ge\eps^{-13}$. (This is possible unless $E=\emptyset$.) Note that our choice of constants implies $K/C\le 1+2\eps^{13}$. Moreover, let $d_{\tilde\omega}^*(A,B)$ be defined as above. The definition of $\tilde\omega$ implies
\begin{align}  \label{eq:weight-rel:1}
 \CST  \deg_\omega(x) &\le \deg_{\tilde\omega}^*(x) \le \CST \deg_\omega(x) + |N(x,B)|\text{ and}\\ \label{eq:weight-rel:2}
 \CST^2 \deg_\omega(x,y) &\le \deg_{\tilde\omega}^*(x,y) \le \CST^2\deg_\omega(x,y) + (2C+1)|N(x,B)\cap N(y,B)|
\end{align}
for all $x,y\in A$. Here the second inequality follows from
\[\lceil C\cdot \omega(x)\rceil^2 \le  \left( C\cdot \omega(x)+1\right)^2 \le C^2\big(\omega(x)\big)^2+2C+1\,.
\]
Moreover,
\begin{align}\label{eq:weight-rel:3}
 \CST  d_\omega(A',B') &\le K d_{\tilde\omega}^*(A',B') \le \CST  d_\omega(A',B') + 1
\end{align}
for all $A'\subseteq A$, $B'\subseteq B$ which in turn implies that 
\begin{align}\label{eq:weight-rel:4}
 \big(\CST\, d_\omega(A',B')\big)^2 - \big(K\,d_{\tilde\omega}^*(A',B')\big)^2 &\le 
1\cdot(\CST+K)
\end{align}
for all $A'\subseteq A$, $B'\subseteq B$.

\medskip
We now verify that conditions~\eqref{eq:reg:w:d-g:1} and~\eqref{eq:reg:w:d-g:2} of Lemma~\ref{lem:reg:weighted:deg-codeg} imply conditions~\eqref{eq:reg:w:CR:1} and~\eqref{eq:reg:w:CR:2} of Lemma~\ref{lem:weighted:CR}. Set $\xi\deff\eps^6$ and let $x\in A$ be such that $|\deg_\omega(x)-d_\omega(A,B)n|\le \eps^{14}n$. It follows from~\eqref{eq:weight-rel:1},~\eqref{eq:weight-rel:3} and the triangle inequality that
\begin{align*}
    |\deg_{\tilde\omega}^*(x) - K\, d_{\tilde\omega}^*(A,B)n|&\le |\deg_{\tilde\omega}^*(x) - \CST \deg_\omega(x)|\\
 &\qquad + |\CST\, \deg_\omega(x) - \CST\, d_\omega(A,B)n| \\
 &\qquad + |\CST\, d_\omega(A,B)n - K\, d_{\tilde\omega}^*(A,B)n|\\
&\le n + \CST  \eps^{14}n + n \le 3n\\
&\le K \xi^2 n\,.
\end{align*}
Hence, condition~\eqref{eq:reg:w:d-g:1} implies condition~\eqref{eq:reg:w:CR:1}.

\medskip

Now let $\{x,y\}\in\binom A2$ satisfy $|\deg_\omega(x,y)-d_\omega(A,B)^2n| < \eps^{9}n$. It follows from~\eqref{eq:weight-rel:2},~\eqref{eq:weight-rel:3} and~\eqref{eq:weight-rel:4} that
\begin{align*}
\left|\deg_{\tilde\omega}^*(x,y)-K^2d_{\tilde\omega}^*(A,B)^2n\right| &\le |\deg_{\tilde\omega}^*(x,y) - \CST^2\deg_{\omega}(x,y)|\\
&\qquad + |\CST^2 \deg_{\omega}(x,y) - \CST^2 d_\omega(A,B)^2n|\\
&\qquad + \left|\CST^2 d_\omega(A,B)^2n - K^2d_{\tilde\omega}^*(A,B)^2n\right|\\
&< (2\CST+1)n + \CST^2 \eps^9n + (C+K)n\\
&\le K^2 \xi n\,,
\end{align*}
where the last inequality is due to $C\le K/\eps$. Thus, condition~\eqref{eq:reg:w:d-g:2} implies condition~\eqref{eq:reg:w:CR:2}.

\medskip

We conclude that $G=(A\dcup B,\tilde\omega)$ satisfies the requirements of Lemma~\ref{lem:weighted:CR}. Hence every $A'\subseteq A$ with $|A'|\ge \eps n$ and every $B'\subseteq B$ with $|B'|\ge \eps n$ has
\[ 
  |d_{\tilde\omega}^*(A',B')-d_{\tilde\omega}^*(A,B)|\le 2\frac{\xi^2}\eps + \frac{\sqrt{5\xi}}{\eps^2-\eps\xi^2}\le \frac52 \eps \,.
\]
Together with~\eqref{eq:weight-rel:3} and the fact that $K/C\le 1+2\eps^{13}$ this finishes the proof as we have
\begin{align*}
  |d_\omega(A',B')-d_\omega(A,B)| &\le  |d_\omega(A',B') - \tfrac KC d_{\tilde\omega}^*(A',B')| \\
& \qquad + |\tfrac KC d_{\tilde\omega}^*(A',B') - \tfrac KC d_{\tilde\omega}^*(A,B)| \\
& \qquad + |\tfrac KC d_{\tilde\omega}^*(A,B) - d_\omega(A,B)|\\
&\le \tfrac 1C + \tfrac KC \tfrac52\eps +\tfrac 1C\\
&\le 3\eps\,.
\end{align*}
\end{proof}

We want to point out that the requirement that $\omega$ is at least $\eps$ does not cause any problem when we apply Lemma~\ref{lem:reg:weighted:deg-codeg} because one could simply increase the weight of all vertices $x$ with $\omega(x)<\eps$ to $\eps$ without changing the weighted densities in the subpairs by more than $\eps$. Hence a graph with an arbitrary weight function is weighted $2\eps$-regular if the graph with the modified weight function is weighted $\eps$-regular.

\medskip

The remainder of this section is dedicated to the proof of Lemma~\ref{lem:reg:weighted:matching} which we restate here.

\begin{lemNN}[Lemma~\ref{lem:reg:weighted:matching}]
Let $\eps>0$ and let $G=(A\dcup B,E)$ with $|V_i|=n$ and weight function $\omega: A \to [\sqrt\eps,1]$ be a weighted $\eps$-regular pair. If $\deg(x)> 2\sqrt\eps n$ for all $x\in A\cup B$ then $G$ contains a perfect matching.
\end{lemNN}

\begin{proof}[of Lemma~\ref{lem:reg:weighted:matching}]
In order to prove that $G=(A\dcup B,E)$ has a perfect matching, we will verify the K\"onig--Hall criterion for $G$, i.e., we will show that $|N(S)| \ge |S|$ for every $S\subseteq A$. We distinguish three cases.

\underline{Case 1, $|S| < \eps n$}: 
The minimum degree $\deg(x) \ge 2\sqrt\eps n$ implies $|N(S)|\ge 2\sqrt\eps n\ge |S|$ for any non-empty set $S$. 

\underline{Case 2, $\eps n \le |S| \le (1-\eps) n$}:
Note that $\deg(x)> 2\sqrt\eps n$ and $\omega(x)\ge \sqrt\eps$ for all $x\in A$ implies that $d_\omega(A,B)> 2\eps$. We now set $T=B\setminus N(S)$. Since $d_\omega(S,T)=0$ and $(A,B)$ is a weighted-$\eps$-regular pair with weighted density greater than $2\eps$ we conclude that $|T|< \eps n$. 

\underline{Case 3, $|S| > (1-\eps)n$}: 
For every $y\in B$ we have $|S|+|N(y)|\ge (1-\eps+2\sqrt\eps)n>n$ and thus $N(y)\cap S \neq \emptyset$. It follows that $N(S)=B$ if $|S|>(1-\eps)n$.
\end{proof}

\subsection*{Chernoff type bounds} 


The analysis of our randomised greedy embedding (see Section~\ref{sec:rga:definition}) repeatedly uses concentration results for random variables. Those random variables are the sum of Bernoulli variables. 
%
If these are mutually independent we use a Chernoff bound (see, e.g.,~\cite[Corollary 2.3]{RandomGraphs}).

\begin{thm}[Chernoff bound] \label{thm:ChernoffBound}
Let $\cA=\sum_{i\in[n]}\cA_i$ be a binomially distributed random variable with $\PP[\cA_i]=p$ for all $i\in[n]$. Further let $c\in[0,3/2]$. Then 
\[ \PP[|\cA - pn| \ge c\cdot pn] \le \exp\left(-\frac{c^2}{3}pn\right).\]
\end{thm}

However, we also consider scenarios where the Bernoulli variables are not independent. 

\begin{lemNN}[Lemma~\ref{lem:pseudo:Chernoff}]
Let $0\le p_1\le p_2 \le 1$, $0<c \le 1$. Further let $\cA_i$ for $i\in[n]$ be a 0-1-random variable and set $\cA:=\sum_{i\in[n]}\cA_i$. If 
\begin{align} \label{eq:pseudo:Independence}
   p_1 \le 
   \PP\left[\cA_i=1 \,\left|~\parbox{150pt}{ $\cA_{j}=1$ for all $j \in J$ and\\ $\cA_j=0$ for all $j\in [i-1]\setminus J$}\right. \right] \le p_2
\end{align}
for every $i\in[n]$ and every $J\subseteq[i-1]$ then
\[
  \PP[\cA \le (1-c) p_1n] \le \exp\left(-\frac{c^2}3 p_1n \right)
\]
and
\[
  \PP[\cA \ge (1+c)p_2n] \le \exp\left(-\frac{c^2}3 p_2n \right)\, .
\]
\end{lemNN}


The somewhat technical conditioning in~\eqref{eq:pseudo:Independence} allows us to bound the probability for the event $\cA_i=1$ even if we condition on any outcome of the events $\cA_j$ with $j<i$. 

The idea of the proof now is to relate the random variable $\cA$ to a truly binomially distributed random variable and then use a Chernoff bound.

\begin{proof}[of Lemma~\ref{lem:pseudo:Chernoff}]
For $k,\ell\in \N_0$ define $a_{\ell,k}=\PP[\sum_{i\le\ell}\cA_i \le k]$ and $b_{\ell,k}=\PP[B_{\ell,p_1}\le k]$ where $B_{\ell,p_1}$ is a binomially distributed random variable with parameters $\ell$ and $p_1$. So both $a_{\ell,k}$ and $b_{\ell,k}$ give a probability that a random variable (depending on $\ell$ and $p_1$) is below a certain value $k$. The following claim relates these two probabilities.

\begin{claim} \label{cl:pseudo:Chernoff:Compare}
For every $k\ge 0$, $\ell\ge 0$ we have $a_{\ell,k}\le b_{\ell,k}$.
\end{claim}

\begin{claimproof}[]
We will prove the claim by induction on $\ell$. For $\ell=0$ we trivially have $a_{0,k} = 1 = b_{0,k}$ for all $k\ge 0$. 
Now assume that the claim is true for $\ell-1$ and every $k\ge 0$. Now
\[
  a_{\ell,0} \le (1-p_1)a_{\ell-1,0}\le (1-p_1)b_{\ell-1,0} = b_{\ell,0}\,.
\]
As
\[ 
\PP\left[\cA_\ell=1 \,\left|~\parbox{150pt}{ $\cA_{j}=1$ for all $j \in J$ and\\ $\cA_j=0$ for all $j\in [\ell-1]\setminus J$}\right. \right] \ge p_1
\] 
for every $J\subseteq [\ell-1]$ it follows that for $k\ge 1$
\begin{align} \label{eq:pseudo:Chernoff:Compare:1}
a_{\ell,k} &\le a_{\ell-1,k-1} + (a_{\ell-1,k}-a_{\ell-1,k-1})(1-p_1)\, .
\end{align}
This upper bound on $a_{\ell,k}$ implies that for every $k\ge 1$ we have
\begin{align*}
 a_{\ell,k} &\leByRef{eq:pseudo:Chernoff:Compare:1} p_1\cdot a_{\ell-1,k-1} + (1-p_1) \cdot a_{\ell-1,k} \\ 
 &\le p_1\cdot b_{\ell-1,k-1} + (1-p_1)\cdot b_{\ell-1,k}\\ 
 &= p_1 \PP[B_{\ell-1,p_1}\le k-1] + (1-p_1) \PP[B_{\ell-1,p_1} \le k]\\
 &= \PP[B_{\ell-1,p_1}\le k-1] + (1-p_1) \PP[B_{\ell-1,p_1} = k]\\
 &= \PP[B_{\ell,p_1}\le k] = b_{\ell,k}\,.
\end{align*}
Here the second inequality is due to the induction hypothesis. This finishes the induction step and the proof of the claim.
\end{claimproof}

Now the first inequality of the lemma follows immediately. We set $\ell=n$, $k=(1-c)p_1n$ and obtain
\begin{align*}
\PP[\cA\le (1-c)p_1n] = a_{n,(1-c)p_1n}
&\le b_{n,(1-c)p_1n} = \PP[B_{n,p_1}\le (1-c)p_1n] \le \exp\left(-\frac{c^2}3p_1n\right),
\end{align*}
where the last inequality follows by Theorem~\ref{thm:ChernoffBound}.

The second assertion of the lemma follows by an analogous argument: set $a_{\ell,k}=\PP[\sum_{i\le\ell}\cA_i \ge k]$ and $b_{\ell,k}=\PP[B_{\ell,p_2}\ge k]$ and obtain
\begin{align} \label{eq:pseudo:Chernoff:Compare:2}
a_{\ell,k} \le a_{\ell-1,k} + (a_{\ell-1,k-1}-a_{\ell-1,k})p_2\, .
\end{align}
It follows by induction on $\ell$ that
\begin{align*}
 a_{\ell,k} &\leByRef{eq:pseudo:Chernoff:Compare:2} (1-p_2)\cdot a_{\ell-1,k} + p_2\cdot a_{\ell-1,k-1} \\ 
 &\le (1-p_2)\cdot b_{\ell-1,k} +  p_2\cdot b_{\ell-1,k-1}\\ 
 &= \PP[B_{\ell-1,p_2}\ge k] + \PP[B_{\ell-1,p_2} = k-1] p_2\\
 &= \PP[B_{\ell,p_2}\ge k] = b_{\ell,k}\,.
\end{align*}
Once more the second inequality follows from the induction hypothesis. Setting $\ell=n$ and $k=(1+c)p_2n$ and using Theorem~\ref{thm:ChernoffBound} again then finishes the proof.
\end{proof}


In addition we need a similar result with a more complex setup.

\begin{lemNN}[Lemma~\ref{lem:pseudo:Chernoff:tuple}]
Let $0<p$ and $a,m,n\in \N$. Further let $\cI\subseteq\mathcal{P}([n])\setminus\{\emptyset\}$ be a collection of $m$ disjoint sets with at most $a$ elements each. For every $i\in [n]$ let $\cA_i$ be a 0-1-random variable. Further assume that for every $I\in\cI$ and every $k\in I$ we have 
\[
  \PP\left[\cA_k=1 \,\left|~\parbox{150pt}{ $\cA_{j}=1$ for all $j \in J$ and\\ $\cA_j=0$ for all $j\in [k-1]\setminus J$}\right. \right] \ge p
\] 
for every $J\subseteq [k-1]$ with $[k-1]\cap I\subseteq J$. Then
\[
\PP\Big[ \big|\{I\in \mathcal{I}: \cA_{i}=1\text{ for all $i\in
        I$}\}\big|\ge \tfrac{1}{2} p^am\Big] \ge
1-2\exp\Big(-\frac1{12} p^am \Big)\, .
\]
\end{lemNN}

\begin{proof}[of Lemma~\ref{lem:pseudo:Chernoff:tuple}]
Let $p>0$, $a,m,n\in\N$ and $\cI$ be given. We order the elements of $\cI$ as $\cI=\{I_1,\dots,I_m\}$ by their respective largest index. This means, the $I_j$ are sorted such that $j'<j$ implies that there is an index $i_j\in I_j$ with $i<i_j$ for all $i\in I_{j'}$. For $i\in[m]$ we now define events $\cB_i$ as 
\[
  \cB_i \deff \begin{cases} 1 & \text{ if $\cA_j=1$ for all $j\in I_i$,}\\ 0 & \text{ otherwise.}\end{cases}
\]
We claim that the events $\cB_i$ satisfy equation~\eqref{eq:pseudo:Independence} where the probability is bounded from below by $p^a$.

\begin{claim}
For every $i\in[m]$ and $J\subseteq [i-1]$ we have 
\[
  \PP\left[\cB_i=1 \,\left|~\parbox{150pt}{ $\cB_{j}=1$ for all $j \in J$ and\\ $\cB_j=0$ for all $j\in [i-1]\setminus J$}\right. \right] \ge p^a\,.
\]
\end{claim}

\begin{claimproof}[]
 Let $i\in[m]$ and $J\subseteq [i-1]$ be given. We assume that $|I|=a$ for ease of notation. (The proof is just the same if $|I|<a$.) So let $I_i=\{i_1,\dots,i_a\}$ be in ascending order and define $i_0\deff 0$. For $v\in \{0,1\}^{i_k-i_{k-1}-1}$ let $H_k(v)$ be the 0-1-random variable with
\[
  H_k(v) = \begin{cases} 1 & \text{ $\cA_{i_{k-1}+\ell}=v_\ell$ for all $\ell\in[i_k-i_{k-1}-1]$,}\\ 0 & \text{ otherwise.}\end{cases}
\]
The rationale for this definition is the following. The outcome of $\cB_i$ is determined by the outcome of the random variables $\cA_{i_j}$. However, we cannot neglect the random variables $\cA_\ell$ for $\ell \notin I_i$ as the $\cA_\ell$ are not mutually independent. Instead we condition the probability of $\cA_{i_j}=1$ on possible outcomes of $\cA_\ell$ with $\ell< i_j$. Now $H_k(v)=1$ with $v \in \{0,1\}^{i_k-i_{k-1}-1}$ represents one outcome for the $\cA_\ell$ with $i_{k-1}<\ell<i_k$. 
We call the $v\in\{0,1\}^{i_k-i_{k-1}-1}$ the \emph{history} between $\cA_{i_{k-1}}$ and $\cA_{i_k}$. It follows from the requirements of Lemma~\ref{lem:pseudo:Chernoff:tuple} that for any tuple $(v_1,\dots,v_k)\in \{0,1\}^{i_1-1}\times \{0,1\}^{i_2-i_1-1}\times \dots \times \{0,1\}^{i_k-i_{k-1}-1}$ we have 
\begin{align} \label{eq:pseudo:Chernoff:tuple:1}
 \PP\left[\cA_{i_k}=1 \left| ~ \parbox{150pt}{ $\cA_{i_j}=1$ for all $j\in[k-1]$ and\\ $H_j(v_j)=1$ for all $j\in[k-1]$} \right. \right] \ge p\,.
\end{align}
However, we are not interested in every possible history $\left(v_1,\dots,v_k\right)$ as some of the histories cannot occur simultaneously with the event $\cB=1$ where 
\begin{align*}
  \cB=1 &\text{ if and only if } \left[\parbox{140pt}{$\cB_j=1$ for all $j\in J$ and\\ $\cB_j=0$ for all $j\in [i-1]\setminus J$}\right]\,. \\
\intertext{For ease of notation we define the following shortcuts}
 \cH(v_1,\dots,v_k)=1 &\text{ if and only if } H_j(v_j)=1 \text{ for all $j\in [k]$,}\\
 \cA^{(k)}=1 &\text{ if and only if } \cA_{i_j}=1 \text{ for all $j\in [k]$.}
\end{align*}
Moreover, we define $C_k$ to be the set of all tuples 
$(v_1,\dots,v_k) \in \{0,1\}^{i_1-1}\times \{0,1\}^{i_2-i_1-1}\times \dots \times \{0,1\}^{i_k-i_{k-1}-1}$ with 
\[
  \PP\left[ \cH(v_1,\dots,v_k)=1 \text{ and } \cB=1 \right] >0\,.
\]
In other words, the elements of $C_k$ are those histories that are compatible with the event that we condition on in the claim. Note in particular that $\cB=1$ if and only if there is a $(v_1,\dots,v_a)\in C_a$ with $\cH(v_1,\dots,v_a)=1$. With these definitions we can rewrite the probability in the assertion of our claim as
\begin{align*}
 \PP\left[\cB_i=1 \mid \cB=1 \right] = \sum_{(v_1,\dots,v_a)\in C_a} \PP\left[\parbox{90pt}{ $\cA^{(a)}=1$ and \\$\cH(v_1,\dots,v_a)=1$} \,\Big|~\cB=1~\Big.\right]\,.
\end{align*}
We now prove by induction on $k$ that 
\begin{align}
P_k\deff\sum_{(v_1,\dots,v_k)\in C_k} \PP\left[\parbox{90pt}{ $\cA^{(k)}=1$ and \\$\cH(v_1,\dots,v_k)=1$} \,\Big|~\cB=1~\Big.\right] \ge p^k
\end{align}
for all $k\in[a]$. The induction base $k=1$ is immediate from the requirements of the lemma as
\[
 P_1 = \sum_{v_1\in C_1} \PP\left[\parbox{65pt}{ $\cA^{(1)}=1$ and \\$\cH(v_1)=1$} \,\Big|~\cB=1~\Big.\right] \geByRef{eq:pseudo:Chernoff:tuple:1} p  \sum_{v_1\in C_1} \PP\left[\cH(v_1)=1 \,\big|~\cB=1~\big.\right] = p\,.
\]
The last equality above follows by total probability from the definition of $C_1$. So assume that the induction hypothesis holds for $k-1$. Then
\begin{align*}
P_k &= \sum_{(v_1,\dots,v_k)\in C_k} \PP\left[\parbox{90pt}{ $\cA^{(k)}=1$ and \\$\cH(v_1,\dots,v_k)=1$} \,\Big|~\cB=1~\Big.\right] \\
 &= \sum_{(v_1,\dots,v_k)\in C_k} 
 \PP\left[\cA_{i_k}=1 \, \Big|
\parbox{130pt}{$\cB=1$ and $\cA^{(k-1)}=1$ and \\$\cH(v_1,\dots,v_k)=1$} \Big. \right] \cdot 
\PP\left[\parbox{90pt}{$\cA^{(k-1)}=1$ and \\$\cH(v_1,\dots,v_k)=1$} \Big| \cB=1 \Big. \right] \\
 &\geByRef{eq:pseudo:Chernoff:tuple:1} p \cdot \sum_{(v_1,\dots,v_k)\in C_k}
\PP\left[\parbox{90pt}{$\cA^{(k-1)}=1$ and \\$\cH(v_1,\dots,v_k)=1$} \Big| \cB=1 \Big. \right] \\
 &= p \cdot
 \sum_{(v_1,\dots,v_{k-1})\in C_{k-1}} 
 \PP\left[\parbox{100pt}{$\cA^{(k-1)}=1$ and \\$\cH(v_1,\dots,v_{k-1})=1$} \Big| \cB=1 \Big. \right] \\
 &= p \cdot P_{k-1} \ge p^k\,.
\end{align*}
The claim now follows as 
\[
  \PP[\cB_i=1 \mid \cB=1] = \sum_{(v_1,\dots,v_a)\in C_a} \PP\left[\parbox{90pt}{ $\cA^{(a)}=1$ and \\$\cH(v_1,\dots,v_a)=1$} \,\Big|~\cB=1~\Big.\right] \ge p^a\,.
\]
\end{claimproof} 
We have seen that the $\cB_i$ are pseudo-independent and that they have probability at least $p^a$ each. Thus we can apply Lemma~\ref{lem:pseudo:Chernoff} and derive
\[
  \PP\left[ |\{i\in[m]: \cB_i=1\}| \ge \frac12 p^a m \right] \ge 1-2\exp\left(-\frac1{12}p^a m\right)\,.
\]
\end{proof}

\end{document}